\documentclass[10pt]{amsart}
\usepackage{latexsym, graphics, amsmath, amssymb,enumerate}
\usepackage{version}
\usepackage{fullpage}
\usepackage{epsfig,graphics}
\usepackage{color}
\usepackage{amsthm, amsopn, amsfonts}

\newcommand{\ret}{\vspace{.1in}}

\newcommand{\apsi}{{\tilde \psi}}
\newcommand{\aphi}{{\tilde \phi}}

\newcommand{\tstar}{ {\tilde{\star}} }
\newcommand{\tnabla}{ {\widetilde{\nabla}} }

\newcommand{\Lp}[2]{\left\Vert \, #1 \, \right\Vert_{#2}}
\newcommand{\lp}[2]{\Vert \, #1 \, \Vert_{#2}}

\newcommand{\llp}[2]{ {|\!|\!| \, #1 \, |\!|\!|}_{#2} }

\newcommand{\td}{\widetilde}

\newcommand{\sd}{ {{\slash\!\!\! d  }} }
\newcommand{\sxi}{ {{\slash\!\!\! \xi  }} }
\newcommand{\ud}{ {\underline{d}} }
\newcommand{\sF}{ {\, {\slash\!\!\!\! F  }} }

\newcommand{\snabla}{\, {\slash\!\!\!\! \nabla} }

\newcommand{\sDelta}{\, {\slash\!\!\!\! \Delta} }

\newcommand{\sdiv}{ \  {\slash\!\!\!\!\!\!\hbox{div}}\, }
\newcommand{\scurl}{ \ {\slash\!\!\!\!\!\!\hbox{curl}}\, }

\newcommand{\bL}{ {\underline{L}} }
\newcommand{\bchi}{ {\underline{\chi}} }

\newcommand{\step}[2]{\noindent\textbf{Step #1:}(\emph{#2})}
\newcommand{\sstep}[1]{\noindent\textbf{Step #1:} }

\newtheorem{theorem}{Theorem}

\newtheorem{lemma}[theorem]{Lemma}

\newtheorem{proposition}[theorem]{Proposition}
\newtheorem{deff}[theorem]{Definition}
\newtheorem{remark}[theorem]{Remark}

\newcommand{\R}{{\mathbb R}}

\newcommand{\diag}{{\text {diag}}}

\newcommand{\la}{\langle}
\newcommand{\ra}{\rangle}

\renewcommand{\S}{{\mathbb S}}

\begin{document}

\title[Local energy decay for Maxwell]
{Local energy decay for Maxwell fields part I: 
Spherically symmetric black-hole backgrounds}

\author{Jacob\ Sterbenz}
\address{Department of Mathematics, University of California San
  Diego, La Jolla, CA 92093-0112} \email{jsterben@math.ucsd.edu}

\author{Daniel\ Tataru} \address{Department of Mathematics, The
  University of California at Berkeley, Evans Hall, Berkeley, CA
  94720, U.S.A.}  \email{tataru@math.berkeley.edu}

\thanks{The first author was partially supported by the NSF grant
DMS-1001675.
The second author  was partially supported by the NSF grant 
DMS-0801261 and by the Simons Foundation.}

\maketitle 
\tableofcontents

\section{Introduction}

Let $(\mathcal{M},g_{\alpha\beta})$ be a 4 dimensional Lorentzian
manifold. In this paper we study the dispersive properties of
two-forms $F_{\alpha\beta}$ in terms of their divergences when
$(\mathcal{M},g_{\alpha\beta})$ is both stationary and spherically
symmetric, and satisfies additional local and global assumptions which are
natural generalizations of the special properties held by
asymptotically flat solutions for the Einstein vacuum equations
$R_{\alpha\beta}=0$. Specifically, we make the following definition:

\begin{deff}[Spherically symmetric stationary ``black holes'']\label{black_hole_defn}
Let $\mathcal{M}\approx \mathbb{R}^2\times \mathbb{S}^2$ be a spherically symmetric 
Lorentzian manifold with metric:
\begin{equation}
		g \ = \ h_{ab}dx^adx^b + r^2 \delta_{AB}dx^A dx^B \ , 
		\qquad a=0,1, \quad A=2,3, 
		\label{metric}
\end{equation}
where $\delta_{AB}$ is the standard round metric on $\mathbb{S}^2$, and where the 2 dimensional
Lorentzian metric $h_{ab}$ depends only on the $(x^a,x^b)$ variables. 
Denote by $\langle\cdot,\cdot \rangle$ the inner product of $g$.
Then we say that $(\mathcal{M},g)$
is a (generalized) ``spherically symmetric stationary black hole'' if the following conditions hold:
\begin{enumerate}[i)]
		\item (Stationary asymptotic flatness) There exists a spherically symmetric
			time function $t$ defined on all of $\mathcal{M}$, 
			i.e.~ $t=t(x^a)$ and $\la dt, dt \ra<0$,
			such that $(t,r)$ forms a system of coordinates in the $x^A=const.$ plane  and
			one has $h_{ab}=\hbox{diag}(-1,1)+ O(r^{-1})$ as $r\to\infty$. Furthermore,
			in this system of coordinates one has
			$\partial_t h_{ab}\equiv 0$ and $\partial_r^k h_{ab}=O(r^{-1-k})$
			for $k\geqslant 1$.
		\item \label{r_M_assumption}
			(Non-degenerate forward global hyperbolicity) There exists  a value $r_0>0$ such that 
			$r=r_0$ is also space-like, and the sign of $g^{rr}=\la dr, dr \ra$ changes only  
			once in $r\geqslant r_0$. We define the value $r=r_M$ to be the unique hypersurface
			where $g^{rr}=0$.
			In addition we assume $\partial_r g^{rr}$ never vanishes.
		\item (Strictly hyperbolic trapping) There is a one and only one value 
		         $r_\mathcal{T}$ in the region
			where $\la dr , dr \ra>0$ such that the time-like surface 
			$r=r_\mathcal{T}$ is a trapped  set for all null geodesics initially tangent to it. 
			Furthermore, this trapped set is normally hyperbolic in the sense 
			of dynamical systems.\footnote{See  Lemma \ref{exterior_lemma} 
			below for a precise formulation of this condition. } 
\end{enumerate}
\end{deff}

A number of remarks about this definition are in order:

\begin{remark}\label{BH_def_rem}
\begin{enumerate}[a)]
	\item The  time  coordinate of  part i)   above is  regular and should not
	be confused with the (singular) Schwarzschild $t$ coordinate in the standard metric
	$ds^2=-(1-\frac{2M}{r})dt^2 + (1-\frac{2M}{r})^{-1}dr^2+r^2d\omega^2$.
	\item \label{d_t_rem}
	The assumptions i) and ii) above imply that $\partial_t$ (computed in $(t,r)$ coordinates)
	is timelike on $r>r_M$,
	null for $r=r_M$, and space-like on $r<r_M$. This can immediately be seen via Cramer's
	rule which gives $\langle\partial_t,\partial_t\rangle =\det(h)g^{rr}$ where $\det(h)\approx -1$
	because it is a regular $(1+1)$ Lorentzian metric in the $(t,r)$ plane. In particular
	this shows that the phenomena of super-radiance is absent for the case of (general) spherically
	symmetric black holes. 
	\item Since $g^{rr}=0$ at $r=r_M$ this hypersurface is null, and by the previous remark
	$\partial_t|_{r=r_M}$ is the null generator. The assumption $\partial_r g^{rr}\neq 0$ implies
	 this null hypersurface enjoys a red shift effect similar to that of Schwarzschild and 
	non-extremal Reissner-Norstr\"om, and this is one of the main stability mechanisms we
	exploit in our work here (we refer the reader to \cite{DR_sch} and \cite{DR_survey} 
	for further information regarding this issue). 
	\item It turns out that condition iii)
	also implies  there are no  trapped null geodesics in the region
	$r>r_M$ other than those tangent to  $r=r_\mathcal{T}$ (e.g.~no
	null geodesics which oscillate in a band $r_1<r<r_2$). Furthermore, it turns out this
	condition has a simple geometric description which is
	$V'(r_\mathcal{T})=0$ and $V''(r_\mathcal{T})<0$, where we have set
	$V=-r^{-2}\la\partial_t,\partial_t \ra_g $. 
\end{enumerate}
\end{remark}

Now orient $(\mathcal{M},g_{\alpha\beta})$ so that $(dt,dr,dx^2,dx^3)$ is a positive
basis of $T^*(\mathcal{M})$, where $r=const$ is given the outward (i.e.~towards $r\to \infty$)
orientation on each 3-surface $t=const$. Recall that given such an oriented Lorentzian
manifold there is a unique isomorphism $\star:\Lambda^p\to \Lambda^{4-p}$
such that $\la \omega ,  \sigma\ra_g dV_g =\omega\wedge \star\sigma$ where
$dV_g=\sqrt{|g|} dt\wedge dr\wedge dV_{\mathbb{S}^2}$, and $dV_{\mathbb{S}^2}$
is the standard (outward oriented) volume form on $\mathbb{S}^2$.

Let $F_{\alpha\beta}$ be an antisymmetric two-tensor  on $\mathcal{M}$, and   define 
$F^{\star}=\star F$.  We  label its divergences as follows:
\begin{equation}
		\nabla^\beta F^{\star}_{\alpha\beta} \ = \ I_\alpha \ , \qquad\qquad
		\nabla^\beta F_{\alpha\beta} \ = \ J_\alpha \ . \label{dsource_Maxwell}
\end{equation}
We call $I_\alpha$ the \emph{magnetic source} and $J_\alpha$ the \emph{electric source} of $F_{\alpha\beta}$.
On physical grounds one usually sets $I\equiv 0$, but for mathematical purposes we
will not do so here. If both $I\equiv 0$ and 
$J\equiv0$, we say $F_{\alpha\beta}$ is a \emph{free Maxwell field}.

It is well known that the first order system \eqref{dsource_Maxwell}
is hyperbolic, so it makes sense to estimate the values of
$F_{\alpha\beta}$ in the wedge $t\geqslant 0$ and $r\geqslant r_0$ in
terms of the values of $(I,J)$ in $t\geqslant 0$ and $r\geqslant r_0$
and the initial restriction  of components $F_{\alpha\beta}|_{t=0}$. Based on experience
with the scalar wave equation on backgrounds such as
$(\mathcal{M},g_{\alpha\beta})$, see e.g.~\cite{MMTT_sch} and
\cite{TT_LE}, one would expect to prove local energy decay estimates
which are roughly of the form\footnote{In the sequel we will prove a
  scale invariant version which corresponds to $\epsilon=0$. In
  addition one needs a logarithmic weight on the trapped set
  $r=r_\mathcal{T}$. The precise version will be given shortly.}
\begin{equation}
		\lp{r^{-\frac{1}{2}-\epsilon} F}{L^2(dV_g)[0,T]} \ \leqslant \ C_\epsilon\big(
		\Lp{F|_{t=0}}{L^2(dV_{g|_{t=0}})} +
		\lp{r^{\frac{1}{2}+\epsilon}I}{L^2(dV_g)[0,T]} 
		+ \lp{r^{\frac{1}{2}+\epsilon} J}{L^2(dV_g)[0,T]}\big) \ , \label{fake_LE}
\end{equation}
where the components of $F,I,J$ are written with respect to regular
basis such as $(\partial_a,r^{-1}\partial_A)$.  Unfortunately, such a
naive estimate is false due to the presence of finite energy bound
state solutions to \eqref{dsource_Maxwell}.  These arise as follows:

Let $\mathcal{S}\subseteq \mathcal{M}$ be a compact space-like two surface
homotopic to some sphere $t=const,r=const$, and define
the quantities (in terms of pull-backs of $F$):
\begin{equation}
		Q_\mathcal{S} \ = \ \int_\mathcal{S} F|_{\mathcal{S}} \ , \qquad
		Q^{\star}_{\mathcal{S}} \ = \ -\int_\mathcal{S} F^{\star}|_{\mathcal{S}} \ , \label{charges}
\end{equation}
where we give the integrals the outward orientation induced by $dV_g$
by first restricting $t$ and then $r$. We call $Q^{\star}_\mathcal{S}$
and (resp $Q_{\mathcal{S}}$) the \emph{electric} (resp
\emph{magnetic}) charge contained within $\mathcal{S}$. Let
$\mathcal{S}'$ be some other space-like two surface such that there
exists a tube $\Sigma(\mathcal{S}',\mathcal{S})$ with boundary
$\partial \Sigma(\mathcal{S}',\mathcal{S})=\mathcal{S}' \sqcup
\mathcal{S}$ and the positive outward orientation through
$\mathcal{S}'$. Then by Stokes theorem and the properties of the
$\star$ operator (see equation \eqref{dual_dsource_Maxwell} below) one
has:
\begin{equation}
		Q_\mathcal{S'} - Q_\mathcal{S} \ = \ -\int_{\Sigma(\mathcal{S}',\mathcal{S})}(\star I)|_{\Sigma} \ , 
		\qquad Q^{\star}_\mathcal{S'} - Q^{\star}_\mathcal{S} \ = \ 
		- \int_{\Sigma(\mathcal{S}',\mathcal{S})}(\star J)|_{\Sigma} \ . \label{charge_diff_eqs}
\end{equation}
In particular if $F_{\alpha\beta}$ is a free Maxwell field then $Q_\mathcal{S}$ and
$Q^{\star}_\mathcal{S}$ are constants equal to:
\begin{equation}
		Q_\infty \ = \ \lim_{r\to\infty}\int_{S_{t,r}} F_{\hat A \hat B}(t,r) dV_{\mathbb{S}^2}  \ , \qquad
		Q^{\star}_\infty \ = \ \lim_{r\to\infty}\int_{S_{t,r}} F_{tr}(t,r)dV_{\mathbb{S}^2} \ ,
		\label{charge_at_infinity}
\end{equation}
where $S_{t,r}$ are the spheres $t=const,r=const$, and $e_{\hat A},e_{\hat B}$ is
an orthonormal basis of $S_{t,r}$ with respect to the restriction of $g$.\footnote{It is not hard to 
see that under conditions i) and ii)
of Definition \ref{black_hole_defn} these values do not depend on the choice of $t$. Here 
$e_{\hat A},e_{\hat B}$ are bounded coefficient linear combinations of 
$r^{-1}\partial_2,r^{-1}\partial_3$; we will reserve the notation $e_A,e_B$ to denote an orthonormal basis
with respect to the (fixed) round metric $\delta_{AB}$ on $\mathbb{S}^2$, i.e.~a
bounded coefficient linear combination of 
$\partial_2,\partial_3$.} If either of
these values is non-zero it represents a finite energy non-dispersing component of $F_{\alpha\beta}$, 
i.e.~an obstruction to estimate \eqref{fake_LE}.


\subsection{Statement of the main theorem}

The main result of the present paper is that in the class of finite energy solutions
to \eqref{dsource_Maxwell} the (local) charges \eqref{charges} represent the \emph{only}
obstruction to local energy decay. To state this result properly, it is necessary to 
compute the finite energy stationary free Maxwell fields. These turn out to be spanned
by the pair:
\begin{equation}
		\overline{F}^{electric} \ = \ \frac{1}{4\pi r^{2}}\sqrt{|h|}dt\wedge dr  
		 \ , \qquad 
		\overline{F}^{magnetic} \ = \ 
		\frac{1}{4\pi}\sqrt{|\delta|}dx^2\wedge dx^3 \ . \notag
\end{equation}
If $F$ is a general solution to \eqref{dsource_Maxwell}, we project it
dynamically onto the span of $\overline{F}^{electric}$ and
$\overline{F}^{magnetic}$ as follows:
\begin{equation}
		\overline{F} \ = \ Q^{\star}(t,r)\overline{F}^{electric} + Q(t,r)\overline{F}^{magnetic} \ , 
		\label{charge_average}
\end{equation}
where $Q^{\star}(t,r)$ and $Q(t,r)$ denote the integrals
\eqref{charges} taken over the spheres $S_{t,r}$, and can be viewed 
as the radial components of $F$.

To state our main theorem we need one additional bit of notation concerning our 
choice of norms. Let $\mathcal{R}_j=\{(t,r,x^A,x^B)\, \big| \, r\sim 2^{j}\}$ for $j\in\mathbb{Z}$. 
Let $r_0$ be a fixed value of $r$ such that $\la dr, dr\ra <0$ as in Definition \ref{black_hole_defn}.
Then for a value $T>0$ we set:
\begin{equation}
		\lp{F}{LE[0,T]} \ = \ \sup_j \lp{r^{-\frac{1}{2}}  F}{L^2(\mathcal{R}_j)[r_0,\infty)\times [0,T]} \ , \qquad
		\lp{J}{LE^{\star}[0,T]} \ = \ \sum_j \lp{r^\frac{1}{2} J}
		{L^2(\mathcal{R}_j)[r_0,\infty)\times [0,T]} \ , \label{LE_norms}
\end{equation}
where the integrals are taken with respect to the volume form $dV_g$.
The main result of this paper  is  the following:\ret

\begin{theorem}[Local energy decay for Maxwell fields]\label{main_thm}
Let $F$ be a solution to \eqref{dsource_Maxwell}
on a space-time $(\mathcal{M},g)$ which satisfies the axioms of Definition \ref{black_hole_defn},
and define $\overline{F}$ according to 
\eqref{charges} and \eqref{charge_average}. Then one has the uniform bound:
\begin{equation}
		\lp{(w_{ln})^{-1}(F-\overline{F})}{LE[0,T]} \ \lesssim \ \lp{F(0)}{L^2(dV_{g|_{t=0}})} 
		+ \lp{w_{ln}(I,J)}{LE^*[0,T]}
		\ , \label{main_thm_est}
\end{equation}
where $w_{ln}(r)=(1+\big|\ln|r-r_\mathcal{T}|\big|)/(1+|\ln(r)|)$. 
Here the components of $F,I,J$ are taken with respect to 
any regular coordinate system $(t,r)$ as in Definition
 \ref{black_hole_defn}, and any
normal frame $e_{\hat A},e_{\hat B}$ on the spheres $S_{t,r}$.
\end{theorem}

\ret

\subsection{The charge equations. A preliminary reduction of the main theorem}

In this section we derive some further equations for the charges \eqref{charges}. 
This will allow us to give an alternate version of Theorem \ref{main_thm}, which is what we 
will actually prove in the sequel.

Recall that on a 4 dimensional Lorentian manifold with oriented
volume form ${dV_g}$, the Hodge $\star$ operator
can be written in an arbitrary basis as follows:
\begin{align}
		\star 1 \ &= \ dV_g \ , \quad
		&(\star J)_{\alpha\beta\gamma} \ &= \ 
		-\epsilon_{\alpha\beta\gamma}^{\ \ \ \ \,  \delta}J_\delta
		\ , \qquad\qquad
		(\star F)_{\alpha\beta} \ = \ \frac{1}{2}\epsilon_{\alpha\beta}^{\ \ \ \gamma\delta}F_{\gamma\delta}
		\ , \notag\\
		(\star G)_\alpha \ &=\  -\frac{1}{6}\epsilon_\alpha^{\ \, \beta\gamma\delta}G_{\beta\gamma\delta} 
		\ , \quad
		&\star \omega \ &= \ \frac{1}{24}\epsilon^{\alpha\beta\gamma\delta}\omega_{\alpha\beta\gamma\delta}
		\ , \notag
\end{align}
for 0 through 4 forms, where $\epsilon_{\alpha\beta\gamma\delta}$ are the components of $dV_g$.
One also has $\star^2=(-1)^{p+1} $ on each class of forms $\Lambda^p$ (again
in the 4 dimensional case). 

It will also be useful for us to have 2 dimensional versions of the
above formulas in both the Lorentzian and Riemannian cases, in
particular for the factor metrics $h_{ab}$ and $\delta_{AB}$ of
\eqref{metric}. We label their volume forms and Hodge star operators
by $\epsilon_{ab}$, $\epsilon_{AB}$, $\star_h$, and $\star_\delta$
(resp). In the Riemannian case we have $\star_\delta^2=(-1)^p$ while
in the Lorentzian case $\star_h^2=(-1)^{p+1}$ as usual. The dual of 0
through 2 in is given by:
\begin{equation}
	\star_{h} 1 \ = \ dV_{h} \ , \qquad
	(\star_{h} J)_a \ = \ -\epsilon_{a}^{\ b}J_b \ , 
	\qquad \star \omega \ = \ \frac{1}{2}\epsilon^{ab}\omega_{ab} \ , \notag
\end{equation}
with identical formulas for $\star_\delta$. In the sequel we will
follow the convention that if $I$ or $J$ is some 1 form on
$\mathcal{M}$ then $\star_h$ (resp $\star_\delta$) act on it by
contraction in the $a$ (resp $A$) indices.

Next, we have the $L^2$ adjoint of $d$ with respect to the inner
product $\la \omega,\sigma \ra_{L^2} =\int_\mathcal{M} \omega\wedge
\star\sigma$, which in the 4 dimensional case is simply $d^\star=\star
d\star$. This operator also has a convenient expression in terms of
the connection Levi-Civita connection $\nabla$ of $g$ which is:
\begin{equation}
		d^\star J \ = \ -\nabla^\alpha J_\alpha \ , \qquad
		(d^\star F)_\alpha \ = \ \nabla^\beta F_{\alpha\beta} \ , \qquad
		(d^\star G)_{\alpha\beta} \ = \ -\nabla^\gamma G_{\alpha\beta\gamma} \ , \notag
\end{equation}
on 1 through 3 forms. 

One also has analogs of the previous formulas in the 2D case (both
Riemannian and Lorentzian).  To avoid confusion we set $\ud$ to denote
the exterior derivative in the $(t,r)$ plane, and $\sd$ to denote
exterior differentiation on the spheres $\mathbb{S}^2$ given by
$t=const,r=const$.  Then one also has co-derivative operators
$\ud^{\star}$ and $\sd^{\star}$, with similar formulas to the last
line above, where now the divergences involve the Levi-Civita
connections of $h$ and $\delta$.

Using the above formulas 
one may write the Maxwell equations \eqref{dsource_Maxwell} as follows:
\begin{equation}
		d F^\star \ = \ \star J \ , \qquad\qquad
		d F \ = \ -\star I \ .
		\label{dual_dsource_Maxwell}
\end{equation}
This allows us to give a useful characterization of the charges \eqref{charges}.

\begin{proposition}[The charge equations]\label{charge_prop}
  Let $F_{\alpha\beta}$ solve \eqref{dual_dsource_Maxwell}, and define
  $\overline{F}_{\alpha\beta}$ as on lines \eqref{charges} and
  \eqref{charge_average}. Then one has:
\begin{equation}
		\nabla^\beta \overline{F}^{\star}_{\alpha\beta} \ = \ \overline{I}_\alpha \ , \qquad\qquad
		\nabla^\beta \overline{F}_{\alpha\beta} \ = \ \overline{J}_\alpha \ , \label{averaged_dsource_Maxwell}
\end{equation}
where $\overline{I}_A=\overline{J}_A\equiv 0$ for $A=2,3$, and for $a=0,1$:
\begin{equation}
		\overline{I}_a \ = \ \frac{1}{4\pi} \int_{\mathbb{S}^2} I_a dV_{\mathbb{S}^2} \ , \qquad
		\overline{J}_a \ = \ \frac{1}{4\pi} \int_{\mathbb{S}^2} J_a 
		dV_{\mathbb{S}^2} \ . \label{source_averages}
\end{equation}
Furthermore, if the source components $I_a$ and $J_a$, for $a=0,1$ have sufficiently fast decay as
$r\to\infty$ we have:
\begin{equation}
		Q(t,r) \ = \ Q_\infty + 
		\int_r^\infty\!\!\!\!\int_{\mathbb{S}^2} (\star_h I)_r(t,s) s^2dV_{\mathbb{S}^2} \ , \qquad
		Q^\star(t,r) \ = \ Q^\star_\infty +
		\int_r^\infty\!\!\!\!\int_{\mathbb{S}^2} (\star_h J)_r(t,s) s^2dV_{\mathbb{S}^2}
		\ , \label{charge_equations}
\end{equation}
where $Q_\infty,Q_\infty^\star$ are given on line \eqref{charge_at_infinity}.
\end{proposition}

An immediate consequence of the last proposition is that Theorem \ref{main_thm}
is reduced to the following result. Furthermore, to understand arbitrary 
$F_{\alpha\beta}$ one only needs to add to this the formulas \eqref{charge_equations}.

\begin{theorem}[Local energy decay estimates for chargeless fields]\label{main_thm_red}
Let $F_{\alpha\beta}$ be a solution to \eqref{dsource_Maxwell} in the region
$r\geqslant r_0$ and $0\leqslant t\leqslant T$
with uniformly vanishing charges $Q(t,r)=Q^\star(t,r)\equiv 0$. Then one has:
\begin{equation}
		\int_{\mathbb{S}^2} I_adV_{\mathbb{S}^2} \ = \ 
		\int_{\mathbb{S}^2} J_a dV_{\mathbb{S}^2} \ \equiv \ 0 \ ,
		\qquad a=0,1, \ 
		\notag
\end{equation}
in $r\geqslant r_0$ and $0\leqslant t\leqslant T$ as well, and in addition there is a 
local energy decay estimate:
\begin{equation}
		\lp{(w_{ln})^{-1}F}{LE[0,T]} \ \lesssim \ \lp{F(0)}{L^2(dV_{g|_{t=0}})} 
		+ \lp{w_{ln}(I,J)}{LE^*[0,T]}
		\ , \label{main_thm_est_red}
\end{equation}
with the same notation as in \eqref{main_thm_est} above.
\end{theorem}

\begin{remark}
The equations \eqref{charge_equations} show that if the charges $Q_\infty$ and $Q_\infty^\star$  vanish
(initially) then one may obtain a local energy decay estimate for the full field $F_{\alpha\beta}$ in terms
of initial data $\lp{F(0)}{L^2(dV_{g|_{t=0}})}$, and certain integrals of $I$ and $J$. However, 
to estimate the local energy decay contribution of the averages  $\overline{I},\overline{J}$
requires a RHS  norm with a different scaling. Precisely, we have the estimate
\begin{equation}
		\lp{(w_{ln})^{-1}F}{LE[0,T]}+	\lp{ r \overline{F}}{LE[0,T]} \ \lesssim \ \lp{F(0)}{L^2(dV_{g|_{t=0}})} 
		+ \lp{w_{ln}(I,J)}{LE^*[0,T]} + \lp{r(\overline{I},\overline{J})}{LE^*[0,T]}
		\ . \label{main_thm_est_extra}
\end{equation}
\end{remark}


\subsection{Background and further remarks}

The motivation for Theorem \ref{main_thm} is its relation to  stability problems in 
general relativity, in particular the problem of proving non-linear stability of the Kerr
family of metrics (see \cite{K_survey} and \cite{DR_survey} for a survey). 
It is generally accepted wisdom that the proof of such a stability theorem
will be strongly predicated on the theory of $L^\infty$ decay estimates for 
first order hyperbolic systems. It is also well known that for hyperbolic (and even other
dispersive)  PDE 
the assumption of  
local energy decay  estimates of the form \eqref{main_thm_est} allows one to 
prove much more refined  decay estimates in higher $L^p$ and weighted $L^2$
spaces (e.g.~Strichartz, conformal energy, and uniform $L^\infty$ estimates).

While there has been great progress lately
towards understanding  the local energy decay and higher 
$L^p$  behavior of solutions to the \emph{scalar} wave
equation $\Box_g=\nabla^\alpha\nabla_\alpha$ on black-hole backgrounds
(e.g.~see \cite{BS_uniform}, \cite{DR_sch}, \cite{DR_L00}, \cite{DR_uniform}, \cite{Luk_unifrom_sch},
\cite{Luk_uniform_kerr}, \cite{MMTT_sch}, \cite{MTT_timedep}, \cite{TT_LE}, \cite{T_uniform},
\cite{Y_decay})
relatively little has been done
for the case of higher spin equations such as Maxwell fields. The main estimates we are aware
of to date
are of conformal energy and uniform $L^\infty$ type in the case of Schwarzschild space
due to work of P. Blue 
\cite{B_Maxwell}, and similar estimates in the case of Kerr with $|a|\ll M$ more
recently due to Andersson-Blue \cite{AB_Max}.
In a companion to our paper \cite{MTT_Max}
it is also shown that local energy decay estimates of the form \eqref{main_thm_est}
imply much more refined $L^\infty$ estimates of the form studied in 
\cite{T_uniform} and \cite{MTT_timedep}
for general solutions to \eqref{dsource_Maxwell} on a broad class of non-symmetric
backgrounds. This puts additional emphasis on
generalizing Theorem \ref{main_thm} to a wider class of space-times.

We remark here that the problem of showing local energy 
decay for higher spin fields such as Maxwell's equations is fundamentally
different from the case of scalar fields due to the presence of finite energy bound state solutions
(i.e.~charges). A somewhat better model for this type of behavior would be the scalar wave equation
$\Box_g$ on a black-hole background plus a possibly negative potential; a problem which is completely open.
The method of the present paper is simplified to a certain extent by the happy
coincidence that for Maxwell fields on spherical backgrounds 
the only finite energy bound state solutions can be eliminated
by subtraction of the spherical average from the sources $(I,J)$ and then removing the charge
from $r=\infty$ (this latter process essentially ``removes the charge from initial data'').
In a sequel to this paper we plan to show that an analog of Theorem \ref{main_thm} holds for 
the case of arbitrary small and sufficiently well localized
perturbations of the metric given by Definition \ref{black_hole_defn}.
The proof of this more general result involves a number of complications not present here.

Finally, we remark that it is not hard to check the Schwarzschild and
family of non-extremal\footnote{For recent work on local energy decay
type estimates for scalar fields in the extremal case
we refer the reader to work of Aretakis \cite{Art1} and \cite{Art2}.} 
Reissner-Nordstr\"om metrics all obey the
general requirements of Definition \ref{black_hole_defn}. Furthermore,
if $g_{\alpha\beta}$ satisfies such assumptions, so does any
sufficiently small and well localized stationary spherically symmetric
perturbation of it. In general we believe that small (not necessarily
symmetric or stationary) perturbations of such metrics give rise to a
general class of space-times with ``good dispersive properties'' not
only for $\Box_g$, but for a variety of first and second order wave
equations. We plan to address this in future works.


\subsection{Organization of the paper}

 This paper is organized as follows. In the remainder of this section we give 
a quick proof of Proposition \ref{charge_prop}, followed by some basic notation
that will be useful in the sequel.

In the  Section \ref{energy_sect} we recall the 
basic energy estimates for Maxwell fields. These will allow us to reduce Theorem
\ref{main_thm_red} to a spatially localized from. However, beyond this simple localization,
energy estimates at the level of $F_{\alpha\beta}$ are not used in the remainder of the proof.

Section \ref{1st_red_sect} reduces the proof of estimate \eqref{main_thm_est_red} to the 
spin-zero components of $F_{\alpha\beta}$.
There are  two of these, which represent the two dynamical degrees of freedom in an
electro-magnetic field. The proof of this reduction follows from elliptic Hodge theory and does
not rely on any time dependent analysis.

In Section \ref{2nd_red_sect} we introduce the main estimate of the paper which is contained 
in Theorem \ref{main_LE_thm1}. This is a local energy decay estimate for a certain 
second order wave equation which is satisfied by the spin-zero components of $F_{\alpha\beta}$.
Reduction of first order relativistic systems to such second order wave equations 
is standard in the literature (e.g.~see \cite{Teuk}, \cite{B_Maxwell}, and 
\cite{H_Ein} for the case of the full Einstein equations). However, our work is the first to provide
sharp local energy decay estimates for general inhomogeneous 
solutions to such equations (in the case of Maxwell), with weights that correspond
to local energy decay at the level of the physical energy of the
original fields. Specifically, our estimates have a sharp
$L^2$ scaling as opposed to an $\dot H^1$ scaling, so in particular they contain more information
at low frequencies. To accomplish this one must keep track of a remnant of the tensorial  character
of the original first order system which resides in the (inhomogeneous) second order equation. It
turns out this feature of the second order equation  
is  intimately connected to the issue of estimating ``dynamic charges'' which come from
the source terms $(I,J)$ on RHS\eqref{dsource_Maxwell}.

Sections \ref{LE_section1} and \ref{LE_section2} are the technical heart of the
paper.  In Section \ref{LE_section1} we prove a preliminary local
energy decay estimate for our spin-zero  Teukolsky
equation. This estimate is a direct analog of the local energy decay
estimate in \cite{MMTT_sch}. Our proof combines ideas from
\cite{BS_uniform}, \cite{DR_sch}, and \cite{MMTT_sch}. 
More specifically, we use red-shift estimates similar to those of \cite{DR_sch}
along with certain ``energy estimates''   to obtain a local energy decay
estimate close to $|r-r_M|\ll 1$ with a source error at $r\approx r_M+\epsilon$. We then introduce an
analog of Regge-Wheeler coordinates in $r>r_M$ and prove a local energy decay estimate
in the exterior along the lines of \cite{BS_uniform}, but with sharp norms as introduced in 
\cite{MMTT_sch}. Interestingly, we are able to show that one can directly glue together the exterior
and horizon estimates because the conservation of (degenerate) energy close to $r=r_M$
controls the worst part of the error coming from truncation of the Regge-Wheeler estimate.
In effect this allows one to  work in two convenient  coordinate  systems simultaneously,
one of them singular, and  to still obtain a regular global   estimate.
 
In Section \ref{LE_section2} we upgrade the preliminary scalar local energy decay estimate 
of Section \ref{LE_section1}  to an estimate which has a non-local character and takes into 
account the structure of the source terms $(I,J)$ to the original first order system \eqref{dsource_Maxwell}. 
To accomplish this
we use a certain ``gauge transformation'' of the inhomogeneous spin-zero wave equation which
is analogous to a space-time Coulomb gauge for Maxwell fields. This allows us to directly
estimate the portion of the spin-zero components of $F_{\alpha\beta}$ which depend 
elliptically on the sources $(I,J)$, the remainder being estimated by the dispersive local energy
decay estimate explained above.


In an Appendix we list some auxiliary estimates that come up in the
bulk of the proof. These are isolated for the convenience of the
reader, and so they don't interrupt the flow of the main argument.


\subsection{Proof of Proposition \ref{charge_prop}}
Here we give a quick demonstration. More general formulas which also
imply this proposition will be given in Lemma \ref{grad_iden_lemma}
below.

\step{1}{Proof of formulas \eqref{charge_equations}}
Using the identities \eqref{charge_diff_eqs} we have
\begin{equation}
		Q_\infty  - Q(t,r) \ = \ -\int_{\Sigma_{t,r}} (\star I)|_{\Sigma_{t,r}} \ , \qquad
		Q^\star_\infty  - Q^\star(t,r) \ = \ -\int_{\Sigma_{t,r}} (\star J)|_{\Sigma_{t,r}} \ , \notag
\end{equation}
where $\Sigma_{t,r}=\{(x^0,x^1,x^2,x^3)\, \big| \, x^0=t, \ x^1>r\}$. 
One also has that with the positive orientation:
\begin{equation}
		(\star I)|_{\Sigma_{t,r}} \ = \ -\frac{1}{2} 
		\epsilon_{rAB}^{\ \ \ \ \ a}I_a dr\wedge dx^A\wedge dx^B 
		\ = \ - r^2 (\epsilon_h)_a^{\ b}I_b dr\wedge dV_{\mathbb{S}^2}
		\ , \notag
\end{equation}
with a similar formula for $(\star J)|_{\Sigma_{t,r}}$. The result follows.

\step{2}{Variation of the charge} From equations
\eqref{charge_equations} we compute the exterior derivatives of
$Q,Q^\star$ in the $(t,r)$ variables. Since the formulas
\eqref{charge_equations} are identical, it suffices to compute things
for the magnetic charge $Q$. We claim that:
\begin{equation}
		d Q \ = \ -r^2 \int_{\mathbb{S}^2} (\star_h I)_a dV_{\mathbb{S}^2} dx^a 
		\ . \label{charge_var_eqs}
\end{equation}
For the $\partial_r$ derivative this is immediate from \eqref{charge_equations}.

On the other hand, to compute the derivative with respect to $t$ we use the continuity
equation $d\star I=0$, which follows immediately from \eqref{dual_dsource_Maxwell}.
In terms of covariant derivatives one has:
\begin{equation}
		\nabla^a (r^2 I)_a + \snabla^AI_A \ = \ -\ud^\star (r^2 I ) - \sd^\star I \ =\ 
		0 \ , \notag
\end{equation}
where $\nabla_a$ (resp $\snabla_A$) is the Levi-Civita connection of $h_{ab}$ (resp
$\delta_{AB}$). Using the identity $\ud^{\star}=\star_h \ud \star_h$
we may write the continuity equation in the mixed form: 
\begin{equation}
		\star_h \ud (r^2 \star_h I) \ = \ - \sd^\star I \ . \label{charge_divergence}
\end{equation}
In term of $(t,r)$ derivatives, and with our choice of orientation,  this gives:
\begin{equation}
		\partial_t (r^2 \star_h I)_r \ = \  \partial_r (r^2 \star_h I)_t -
		\sd^\star (\sqrt{|h|} \star_\delta I) \ . \notag
\end{equation}
Since an angular divergence integrates to zero on $\mathbb{S}^2$, 
after applying the fundamental theorem 
of calculus in the radial variable  we are left with:
\begin{equation}
		\partial_t Q \ = \ -r^2 \int_{\mathbb{S}^2} (\star_h I)_t dV_{\mathbb{S}^2} 
		\ , \notag
\end{equation}
which proves \eqref{charge_var_eqs}.

\step{3}{Computing $d\overline{F}$} 
First note that
$d \big( Q^\star(t,r) \overline{F}^{electric}\big)  =  0$ 
for any function $Q^\star(t,r)$. On the other hand, by line 
\eqref{charge_var_eqs} we have:
\begin{equation}
		d \big( Q(t,r) \overline{F}^{magnetic}\big) \ = \ -\frac{1}{4\pi} r^2 
		\big(\star_h \int_{\mathbb{S}^2}  
		I dV_{\mathbb{S}^2}\big)_a \cdot  dx^a\wedge dV_{\mathbb{S}^2} \ = \ 
		-\star \big( \frac{1}{4\pi} \int_{\mathbb{S}^2} I_a dV_{\mathbb{S}^2}\cdot dx^a\big) \ . \notag
\end{equation}
Adding these  identities, and noting the second idenitity of line \eqref{dual_dsource_Maxwell},
gives the second equation on line \eqref{averaged_dsource_Maxwell}.

The first identity on  line \eqref{averaged_dsource_Maxwell} follows from the previous 
calculation  and duality.


\subsection{Some notation}
We list here several notational conventions that will be used in the sequel. 

As usual we denote
$A\lesssim B$ (resp $A\ll B$; $A\approx B$)
if $A\leqslant CB$ for some fixed $C>0$ which may change from line to line
(resp $A\leqslant \epsilon B$ for a small $\epsilon>0$;
both $A\lesssim B$ and $B\lesssim A$).

For a collection of non-commutative operators  we often denote their ordered product 
without parenthesis. A typical example of this would be a mixture of pseudodifferential operators such as
$\partial_r \Delta^{-1} f \partial_r$, where $f$ is a function of $(r,x^A)$.

If $\phi(t,r,x^A)$ is a space-time function we denote by $\phi[t_0]$ the restriction of the gradient $d\phi$
to the hypersurface $t=t_0$. This notation is context sensitive as the components of 
$d\phi$ will be given with respect to different frames depending on the application. In practice there
are only two frames  used here, either $(\partial_t,\partial_r , e_{\hat{A}})$ or
$(\partial_t,\partial_r , e_{{A}})$. The choice (i.e.~normalization of angular derivatives)
will be clear from context.

A basic estimate that will be used many times in the sequel is the following: If $f$ is a sufficiently
smooth function on $\mathbb{S}^2$ we denote by $\overline{f}=(4\pi)^{-1}
\int_{\mathbb{S}^2}f dV_{\mathbb{S}^2}$ its spherical average. Then Poincare's estimate says that:
\begin{equation}
		\lp{f-\overline{f}}{L^2(dV_{\mathbb{S}^2})} \ \lesssim \ 
		\lp{\sd f}{L^2(dV_{\mathbb{S}^2})} \ . \label{S2_poincare}
\end{equation}

\ret

\section{Energy Estimates for Maxwell Fields}\label{energy_sect}

In this section we recall the basic energy estimates for Maxwell
fields on a Lorentzian background. These will be used in a auxiliary
manner in the sequel, but are not necessary for the core part of the
proof of local energy decay.\footnote{Our proof of local energy decay
  uses energy conservation in an essential way, but only for  certain
  components of $F_{\alpha\beta}$. For these specific components we will
  only need to show certain spatially localized versions of the
  conservation of energy.}  For any antisymmetric two-form on a
$(3+1)$ space-time we define:
\begin{equation}
		Q_{\alpha\beta}[F] \ = \ F_{\alpha\gamma}F_{\beta}^{\ \ \gamma}
		- \frac{1}{4}g_{\alpha\beta}F_{\gamma\delta}F^{\gamma\delta} \ = \ 
		\frac{1}{2}\big( F_{\alpha\gamma}F_{\beta}^{\ \ \gamma} + 
		F^\star_{\alpha\gamma}F_{\beta}^{\star \ \gamma} \big) \ . \notag
\end{equation}
Then with the notation of line \eqref{dsource_Maxwell} we have the divergence identity:
\begin{equation}
		\nabla^\alpha Q_{\alpha\beta}[F] \ = \ \frac{1}{2}( F_{\alpha\beta} J^\alpha + 
		F^\star_{\alpha\beta} I^\alpha)  \ . \label{divergence_iden}
\end{equation}
To measure energy contributions  we use the following definition:

\begin{deff}[Regular space-like hypersurfaces]\label{spacelike_defn}
Fix $(t,r)$ coordinates as in Definition \ref{black_hole_defn}.
We say a space-like hypersurface $\mathcal{S}$ in $(\mathcal{M},g_{\alpha\beta})$
is ``regular'' if there exists a constant $c>0$ such that:
\begin{equation}
		T \ = \ b_+\ell + b_-\underline{\ell} \ , \qquad
		b_\pm > c \ , \notag
\end{equation}
where $T$ is the future directed (with respect to $-\nabla t$) normal 
to $\mathcal{S}$ and $\ell$, $\underline{\ell}$,
are future directed null vector-fields such that:
\begin{equation}
		\ell \ = \ -\nabla t + a\partial_r \ , \qquad
		\underline{\ell} \ = \ -\nabla t - a \partial_r \ , \qquad\qquad
		a >0 \ . \notag
\end{equation}
\end{deff}
In the usual way we have the following:

\begin{proposition}[Energy estimates for Maxwell fields]\label{energy_prop}
Let $\mathcal{R}\subseteq \mathcal{M}$ be an open set bounded by two regular
space-like hypersurfaces $\mathcal{C}_1$ and $\mathcal{C}_2$, where 
$\mathcal{C}_2$ is to the future of $\mathcal{C}_1$ (with respect to $-\nabla t$).
If  $F_{\alpha\beta}$ solves \eqref{dsource_Maxwell} on $\mathcal{R}$
then one has the following uniform bound in $\epsilon>0$:
\begin{equation}
		\lp{F|_{\mathcal{C}_2}}{L^2(dV_{g|_{\mathcal{C}_2}})} \ \lesssim \ 
		\epsilon \lp{(w_{ln})^{-1} F}{LE(\mathcal{R})} + 
		\lp{F|_{\mathcal{C}_1}}{L^2(dV_{g|_{\mathcal{C}_1}})} 
		+ \epsilon^{-1}\lp{w_{ln} (I,J)}{LE^*(\mathcal{R})} \ . \label{maxwell_energy}
\end{equation}
Here the components of $F,I,J$ are computed in any regular  frame such as
$\{\partial_t,\partial_r , e_{\hat{A}}, e_{\hat{B}}\}$, and the notation $F|_{\mathcal{C}_i}$
denotes the restriction of such scalar components to $\mathcal{C}_i$ (not the pull-back of forms).
\end{proposition}

\begin{proof}
The proof is standard. First contract $Q_{\alpha\beta}$ above with the vector-field $-\nabla t$.
Then using the divergence identity \eqref{divergence_iden} one employs Stokes theorem and
Cauchy-Schwartz. The key observation is that:
\begin{equation}
		 - Q(T,\nabla t) \ \approx \ \sum_{\alpha<\beta} |F_{\alpha \beta}|^2 \ , \notag
\end{equation}
where $T$ is the (future) normal to $\mathcal{C}_i$ and
the RHS sum is taken over the frame $\{\partial_t,\partial_r , e_{\hat{A}}, e_{\hat{B}}\}$.
See \cite{HE} for more details.
\end{proof}

\ret

\section{First Order Formulation of the Equations. Reduction of the Main Theorem to
a Spin-Zero  Local Energy Decay Estimate}\label{1st_red_sect}

We begin with some first order equations for components of $F_{\alpha\beta}$ when
 \eqref{dsource_Maxwell} holds. First we introduce some notation:
 \begin{equation}
 		\phi \ = \ \frac{1}{2}\epsilon^{AB}F_{AB} \ , \qquad
		\phi^\star \ = \ -\frac{1}{2}\epsilon^{AB}F^\star_{AB} \ , \qquad
		\sF_{aA} \ = \ F_{aA} \, \qquad
		\sF^\star_{aA} \ = \ F^\star_{aA}
		\ . \label{components}
 \end{equation}
 Note that the quantities $\phi,\phi^\star$ are scalars.
Here we think of $\sF,\sF^\star$ as  sections of $T^*(\mathbb{R}^2)\otimes T^*(\mathbb{S}^2)$, in other
words $\sF$ is a tensor with four components $F_{02},F_{03},F_{12},F_{13}$ and similarly for $\sF^\star$. 
We let  $\star_h,\ud,\ud^\star$ (resp $\star_\delta,\sd,\sd^\star$) 
act on $\sF,\sF^\star$ in the obvious way, by touching only the first (resp  second) set of components.
It turns out that by Hodge duality the second tensor $\sF^\star$ is redundant, specifically:
\begin{equation}
		\sF^\star_{aA} \ = \ \frac{1}{2} \epsilon_{aA}^{\ \ \ \gamma\delta}F_{\gamma\delta}
		\ = \ -(\epsilon_h)_{a}^{\ \, b}(\epsilon_\delta)_A^{\ \, B}F_{bB} \ = \ 
		-(\star_h\star_\delta \sF)_{aA} \ . \label{sF_duality}
\end{equation}
In particular notice that the 6 scalar quantities $\phi,\phi^\star$ and $\sF_{aA}$ for $a=0,1$, $A=2,3$
span the values of $F_{\alpha\beta}$. 

\begin{remark}
One may compare the quantities listed on line \eqref{components} to the standard null decomposition
of an electromagnetic field on Minkowski space as defined by Christodoulou-Klainerman \cite{CK_Fields}. 
Then we have $\phi=r^2\sigma$, $\phi^\star=r^2\rho$, and $\sF_{aA}$ is a linear combination
of $r\alpha_A$ and $r\underline{\alpha}_A$. We remark that since all components are on an equal footing
with respect to the natural $L^2$ energy \eqref{maxwell_energy}, there is no need to further decompose
$\sF_{aA}$ in proving our estimates.
\end{remark}

Our main result here is to relate the $(t,r)$ derivatives of $\phi,\phi^\star$ to the angular
derivatives of $\sF$ and the sources $I,J$. This will generalize \eqref{charge_var_eqs} above.

\begin{lemma}[Gradient identities for $\phi,\phi^\star$]\label{grad_iden_lemma}
Let $F_{\alpha\beta}$ satisfy the equations \eqref{dsource_Maxwell}, and define the 
quantities on line \eqref{components} from it. Then one has:
\begin{align}
		\ud \phi \ &= \ \star_\delta \sd \sF - r^2 \star_h I
		\ , \label{dphi1}\\
		 \ud \phi^\star \ &= \ - \star_\delta \sd \sF^\star - r^2\star_h J
		\ , \label{dphi2}\\
		\star_h \ud \phi^\star \ &= \ \sd^\star \sF - r^2 J
		\ . \label{dphi3}
\end{align}
\end{lemma}

\begin{proof}
Notice that line \eqref{dphi2} follows from \eqref{dphi1} and Hodge duality which sends 
$\phi \to-\phi^\star$, $\sF\to \sF^\star$, and $I\to -J$.

Next, we see that line \eqref{dphi3} follows by applying $\star_h$ to line \eqref{dphi2},
using $[\star_h,\star_\delta\sd]=0$, and then using the identity \eqref{sF_duality} with
$\sd^\star=-\star_\delta \sd \star_\delta$.

To show \eqref{dphi1} we expand the second member of line 
\eqref{dual_dsource_Maxwell} in the basis $(\partial_a,\partial_A,\partial_B)$
which gives:
\begin{equation}
		\partial_a  F_{AB} - \snabla_{A}F_{aB} + \snabla_{B}F_{aA} \ = \ -(\star I)_{aAB} \ , \notag
\end{equation}
where $\snabla$ is the standard round connection on $\mathbb{S}^2$. Tracing the last line with 
$\frac{1}{2}\epsilon^{AB}$ and using the identity:
\begin{equation}
		\frac{1}{2}\epsilon^{AB} (\star I)_{aAB} \ = \ -\frac{1}{2}\epsilon^{AB} r^2 \epsilon_a^{\ \, b}
		\epsilon_{AB} I_b \ = \ -r^2 \epsilon_a^{\ \, b}I_b \ = \ 
		r^2(\star_h I)_a \ , \notag
\end{equation}
we have the desired result.
\end{proof}


\subsection{Reduction of the main theorem}

The equations \eqref{dphi1} and \eqref{dphi3} completely determine the values
of $\sF$ in terms of $\phi,\phi^\star$. This may be quantified as follows:

\begin{proposition}[Reduction to spin-zero components]\label{div_curl_prop}
Let $\sF$, $\phi,\phi^\star$, and $I,J$ solve the system \eqref{dphi1} and \eqref{dphi3}.
Suppose in addition that:
\begin{equation}
		\int_{\mathbb{S}^2}\phi dV_{\mathbb{S}^2}
		\ = \ \int_{\mathbb{S}^2}\phi^\star dV_{\mathbb{S}^2}
		\ = \ \int_{\mathbb{S}^2}I_a dV_{\mathbb{S}^2} \ = \ 
		\int_{\mathbb{S}^2}J_a dV_{\mathbb{S}^2} \ \equiv \ 0 
		\ , \qquad a=0,1. \label{spin_zero_moment}
\end{equation}
Then one has the $L^2$ estimate:
\begin{equation}
		\lp{(w_{ln})^{-1} \sF}{LE[0,T]} \ \lesssim \ 
		\lp{(w_{ln})^{-1}r^{-1}(-\sDelta)^{-\frac{1}{2}}(\ud\phi,\ud\phi^\star)}{LE[0,T]} + 
		\lp{(I,J)}{LE^*[0,T]} \ , \label{div_curl_est}
\end{equation}
where $-\sDelta=\sd^\star d$ is the scalar Laplace-Beltrami operator on $\mathbb{S}^2$.
Here the components of $\sF$ are taken in a normalized basis $(\partial_a, r^{-1}\partial_A)$.
\end{proposition}

\begin{proof}
This follows immediately from $L^2$ Hodge estimates on $\mathbb{S}^2$. Recall that the 
operators $\ \sdiv=-\sd^\star$ and $\ \scurl=\star_\delta \sd$ on $T^*(\mathbb{S}^2)$
give rise to a double sided singular integral $L^2$ estimate:
\begin{equation}
		\lp{\omega}{L^2(\mathbb{S}^2)} \ \approx \ 
		\lp{(-\sDelta)^{-\frac{1}{2}}\, \sdiv \omega}{L^2(\mathbb{S}^2)} 
		+ \lp{(-\sDelta)^{-\frac{1}{2}}\, \scurl \omega}{L^2(\mathbb{S}^2)} \ . \label{hodge_equiv}
\end{equation}
Setting $(\omega_a)_A = \sF_{a\hat{A}}$ to be a one form on $\mathbb{S}^2$ indexed by 
$a=0,1$ with values the normalized components of $\sF$, we may write \eqref{dphi1} and \eqref{dphi3}
as:
\begin{equation}
		\sdiv \omega \ = \ -\ r^{-1}\star_h \ud \phi^\star - rJ \ ,
		\qquad
		\scurl \omega \ = \ r^{-1}\ud\phi + r\star_h I \ . \notag
\end{equation}
The estimate \eqref{div_curl_est} follows at once from this, the boundedness of $(-\sDelta)^{-\frac{1}{2}}$
on $L^2(\mathbb{S}^2)$ functions with zero average, and the inclusion $r^{-1}LE^*\subseteq LE$.
\end{proof}

\ret

\section{Second Order Wave Equations and Local Energy Decay Estimates. 
Proof of the First Order Decay  Estimate}\label{2nd_red_sect}

In this section we reduce the estimation of the first term on RHS \eqref{div_curl_est}
to the following Theorem, which one may view as the main technical result of 
the paper. To state it we introduce the following operator, which we call the 
\emph{spin-zero wave equation}:
\begin{equation}
		\Box^0 \ = \ \Box_h + r^{-2}\sDelta \  =\ \nabla^a\nabla_a + r^{-2}\snabla^A\snabla_A
		\ . \label{spin_0_box}
\end{equation}
For this operator we have:

\begin{theorem}[Inverse angular gradient local energy decay estimates]\label{main_LE_thm1}
Let $\phi(x^a,x^A)$ be a scalar function defined 
on the slab $[0,T]\times[r_0,\infty)\times\mathbb{S}^2$.
In addition let $G_{a}(x^a,x^A)$ be a one form in the $x^a$ variables, whose coefficients
also depending 
on $x^A \in \mathbb{S}^2$, such that ${\star}_h \ud G=K$. 
Finally let $H(x^a,x^A)$ be another scalar function
Suppose that all of these objects obey the moment condition:
\begin{equation}
		\int_{\mathbb{S}^2}\phi(x^a) dV_{\mathbb{S}^2} 
		\ = \ \int_{\mathbb{S}^2}G_b (x^a) dV_{\mathbb{S}^2}
		\ = \ \int_{\mathbb{S}^2}H(x^a) dV_{\mathbb{S}^2}
		\ \equiv \ 0 \ , \qquad b=0,1 \ , \label{zero_moment}
\end{equation}
throughout $[0,T]\times[r_0,\infty)$. If $\phi,G,H$ are all supported in $\{r \leqslant CT\}$ and 
\begin{equation}
		\Box^0\phi \ = \ \nabla^a G_a + H \ , \label{spin_zero_eq}
\end{equation}
then one has the local energy decay type estimate:
\begin{equation}
\begin{split}
		\lp{(w_{ln})^{-1} r^{-1}\big(\ud(-\sDelta)^{-\frac{1}{2}}\phi,r^{-1}\phi\big)}{LE[0,T]}
		\ \lesssim  & \
		\lp{(-\sDelta)^{-\frac{1}{2}}r^{-1}\big(\ud \phi(0)- G(0)\big)}{L^2(dV_{g|_{t=0}})}
		+\lp{r^{-2}\phi(0)}{L^2(dV_{g|_{t=0}})} \\
		& + \lp{w_{ln} r^{-1}\big(
		r^{-1}G,(-\sDelta)^{-\frac{1}{2}}K, (-\sDelta)^{-\frac{1}{2}}H\big)}{LE^*[0,T]}
		\ . 
\end{split}
\label{main_est1}
\end{equation}
Here $w_{ln}(r)= (1+\big|\ln|r-r_\mathcal{T}|\big|)/(1+|\ln(r)|)$ as usual.
\end{theorem}

The proof of this theorem will occupy the second portion of the paper. First, we
use it to prove Theorem \ref{main_thm_red}. In light  of the energy 
estimates \eqref{maxwell_energy} the main step is to show the following:

\begin{proposition}[Local energy decay estimates for  spin-zero components]\label{phi_est_prop}
Let $\sF$, $\phi,\phi^\star$, and $I,J$ satisfy the assumptions of Proposition \ref{div_curl_prop},
and in addition suppose each of these quantities is supported in the region $r<CT$ for
$C$ a sufficiently large fixed constant. Then one has the uniform (in $T$) bound:
\begin{multline}
		\lp{(w_{ln})^{-1} r^{-1}(-\sDelta)^{-\frac{1}{2}}(\ud\phi,\ud\phi^\star)}{LE[0,T]} 
		+ \lp{(w_{ln})^{-1} r^{-2}(\phi,\phi^\star)}{LE[0,T]}
		\ \lesssim \
		\lp{F|_{t=0}}{L^2(dV_{g|_{t=0}})}\\ 
		+ \lp{w_{ln} (I,J)}{LE^*[0,T]} + \lp{\big(\star_h I_0 (r_0),\star_h J_0 (r_0)\big)}
		{L^2(dV_{\mathbb{S}^2} dt )[0,T]} \ . \label{spin_zero_LE}
\end{multline} 
\end{proposition}

\begin{proof}[Proof that Theorem \ref{main_LE_thm1} implies Proposition \ref{phi_est_prop}]
This boils down to a direct algebraic calculation, whose point is to derive a second 
order equation of the form \eqref{spin_zero_eq} for $\phi,\phi^\star$, where $G_a$ and $H$ 
can be estimated in terms of $I$ and $J$.

\step{1}{Derivation of the second order equation}
All of our computations here are more easily done with respect 
to a conformal metric. Let $\Omega$ be a weight function and set $\td{g}=\Omega^2 g$.
We denote by $\tstar$ the corresponding Hodge operator, which obeys the identity
$\tstar=\Omega^{4-2p}\star$ on each $\Lambda^p$. In particular notice that the quantity
$F^\star$ is conformally invariant. With respect to the new metric the Maxwell system can
be written as (alternatively):
\begin{align}
		d F \ &= \ -\Omega^{-2}\tstar I \ ,   &d F^\star \ &= \ \Omega^{-2}\tstar J \ , \notag\\
		d^{\tstar} F \ &= \ \Omega^{-2} J \ ,  
		&d^{\tstar} F^\star \ &= \ \Omega^{-2} I \ . \notag
\end{align}
Combining these formulas we have:
\begin{equation}
		\Box^{hodge} F \ = \ d^{\tstar}(\Omega^{-2} \tstar I) - d(\Omega^{-2} J) \ , \qquad
		\Box^{hodge} F^{\star} \ = \ -d^{\tstar}(\Omega^{-2} \tstar J) - d(\Omega^{-2} I) \ , \notag
\end{equation}
where $\Box^{hodge}=-(d d^{\tstar} + d^{\tstar}d)$ is the Hodge Laplacian of $\td{g}$. 

In terms of the connection $\tnabla$ of $\td{g}$ we may write $\Box^{hodge}$ as
follows:
\begin{align}
		\Box^{hodge}F_{\alpha\beta} \ &= \ \tnabla^{\gamma}(\tnabla_{\alpha}F_{\beta\gamma}
		+   \tnabla_{\gamma} F_{\alpha\beta }+ \tnabla_{\beta} F_{\gamma\alpha}) - 
		\tnabla_\alpha \tnabla^\gamma F_{\beta\gamma} + 
		\tnabla_\beta \tnabla^\gamma F_{\alpha\gamma} \ , \notag\\
		&= \ \tnabla^\gamma\tnabla_\gamma F_{\alpha\beta} 
		+ [\tnabla_\alpha,\tnabla_\gamma]F^\gamma_{\ \, \beta}
		+ [\tnabla_\beta,\tnabla_\gamma] F_\alpha^{\ \, \gamma} \ , \notag\\
		&=\ \Box_{\td{g}}F_{\alpha\beta} - \td{R}_{\alpha}^{\ \ \gamma}F_{\gamma\beta}
		- \td{R}_{\beta}^{\ \ \gamma}F_{\alpha\gamma}
		- \td{R}_{\alpha\beta}^{\ \ \ \gamma\delta}F_{\gamma\delta} \ , \notag
\end{align}
where $\Box_{\td{g}}$ is the covariant wave equation acting on two-forms, 
$\td{R}_{\alpha\beta\gamma\delta}$ is the Riemann curvature tensor of $\td{g}$, and
$\td{R}_{\alpha\beta}
=\td{g}^{\gamma\delta}\td{R}_{\alpha\gamma\delta\beta}$ is its Ricci curvature.
Here our curvature convention is $[\nabla_\alpha,\nabla_\beta]\omega_\gamma=
R_{\alpha\beta\ \, \gamma}^{\ \ \ \, \delta}\omega_\delta$ for one-forms.

We now choose $\Omega$ in order to simplify the form of $\Box^{hodge}$ above. This is done
by choosing $\Omega=r^{-1}$ so $\td{g}$ becomes a pure 
direct sum metric instead of a warped product. For this new metric the curvatures diagonalize into
pure $a,b$ and $A,B$ components. It is the latter which is important for us here, which is simply
 the Riemann and Ricci curvature of $\mathbb{S}^2$, that is: 
\begin{equation}
		\td{R}_{ABCD} \ = \ \delta_{AD}\delta_{BC} - \delta_{AC}\delta_{BD} \ , \qquad
		\td{R}_{AB} \ = \ \delta_{AB} \ . \notag
\end{equation}
For the D'Alembertian this gives the simple formula:
\begin{equation}
		\Box^{hodge}F_{AB} \ =\  \Box_{\td{g}}F_{AB} \ . \notag
\end{equation}
Tracing this with respect to $\frac{1}{2}\epsilon^{AB}$ gives us:
\begin{equation}
		\Box_{\td{g}}\phi\ = \ \frac{1}{2}\epsilon^{AB}  d^{\tstar}(r^2 \tstar I)_{AB}
		- \frac{1}{2}\epsilon^{AB} d(r^2 J)_{AB} \ , \qquad
		\Box_{\td{g}}\phi^\star\ = \ \frac{1}{2}\epsilon^{AB}  d^{\tstar}(r^2 \tstar J)_{AB}
		+ \frac{1}{2}\epsilon^{AB} d(r^2 I)_{AB} \ , \notag
\end{equation}
where $\Box_{\td{g}}=\tnabla^a\tnabla_a + \sDelta=r^{2}\Box^0$ is now the scalar covariant wave equation of
$\td{g}$.

It remains to compute the traces on the two RHS of the last line above. For the exterior derivative terms
we immediately have:
\begin{equation}
		\frac{1}{2}\epsilon^{AB} d(r^2 I)_{AB} \ = \ r^2 \star_\delta \sd I \ , \qquad
		\frac{1}{2}\epsilon^{AB} d(r^2 J)_{AB} \ = \ r^2 \star_\delta \sd J \ . \notag
\end{equation}
For the co-derivative expressions we compute:
\begin{equation}
		d^{\tstar}(r^2 \tstar I)_{AB} \ = \  \star d(r^2 I)_{AB} \ = \ 
		r^2 \epsilon_{AB}\cdot \frac{1}{2}\epsilon^{ab}\ud(r^2I)_{ab} 
		\ = \ r^2 \epsilon_{AB} \star_h \ud (r^2I) \ = \ -r^2 \epsilon_{AB} \nabla^a (r^2\star_h I)_a 
		\ , \notag
\end{equation}
with an identical formula for $d^{\tstar}(r^2 \tstar J)_{AB}$. All together this yields:
\begin{equation}
		\Box^0\phi \ = \ -\nabla^a(r^2\star_h I)_a - \star_\delta \sd J \ , \qquad
		\Box^0\phi^\star \ = \ -\nabla^a(r^2\star_h J )_a + \star_\delta \sd I \ . \label{Box_0_RHS}
\end{equation}

\step{2}{Application of estimate \eqref{main_est1}}
It suffices to treat the first equation on line \eqref{Box_0_RHS} as the second is of the same form.
To set up for estimate \eqref{main_est1} we define:
\begin{equation}
		G \ := \ -r^2 \star_h I \ , \qquad H \ := \ -\star_\delta \sd J \ , \qquad
		K \ := \ \sd^\star I \ . \notag
\end{equation}
Using the continuity equation \eqref{charge_divergence} we have $\star_h dG=K$ as required.
Notice also the the moment conditions \eqref{zero_moment} are all satisfied thanks to
\eqref{spin_zero_moment}. 

It remains to bound the terms occurring on the RHS of \eqref{main_est1} for these choices in
terms of RHS \eqref{spin_zero_LE}. For the terms $H$ and $K$ this is an immediate 
consequence of \eqref{hodge_equiv}. Notice that the extra factor $r^{-1}$ disappears when one
switches from the $\mathbb{S}^2$ basis $e_A$ to the regular basis $e_{\hat{A}}=r^{-1}e_A$.

Next, the estimate for the one form $G$ is also immediate due to the truncation condition which
implies $(\sqrt{rT})^{-1}\lesssim r^{-1}$. Likewise we have $r^{-2}\phi = F_{\hat{A}\hat{B}}$ in an
ordered $g$ orthonormal basis $e_{\hat A},e_{\hat B}$ of $S_{t,r}$, which gives the desired
bound for the undifferentiated initial data on RHS \eqref{main_est1}.

Finally, we dispense with the gradient terms in the initial data on RHS \eqref{main_est1}. Using
equations \eqref{dphi1} above and the definition of $G$ we have:
\begin{equation}
		\ud\phi(0) - G(0) \ = \ \star_\delta \sd\sF(0) \ , \notag
\end{equation}
and the desired $L^2$ bound follows again from \eqref{hodge_equiv} and writing the components
of $\sF$ in a regular basis.
\end{proof}


\subsection{Proof of the main theorem}

We are now ready to show Theorem \ref{main_thm_red}, which follows easily
by combining estimates \eqref{maxwell_energy}, \eqref{div_curl_est}, and \eqref{spin_zero_LE}.

In order to use \eqref{spin_zero_LE} we must first  dispense with estimate \eqref{main_thm_est_red}
in the region $r>CT$ for some sufficiently large $C>0$. 
To do this we will use \eqref{maxwell_energy} to prove \eqref{main_thm_red}
in a slightly larger region, namely $r>Ct$. The desired estimate for $0\leqslant t\leqslant 1$ follows
easily from local considerations. For larger values of $t$ we can  integrate the LHS of
estimate \eqref{maxwell_energy} over a family of uniformly space-like hyper-surfaces such as
$r=Ct$ which  upon bootstrapping the first term on RHS \eqref{maxwell_energy} gives:
\begin{equation}
		\lp{F}{LE(r>Ct)[0,T]} \ \lesssim \ \lp{F(0)}{L^2(dV_{g|_{t=0}})} 
		+ \lp{(I,J)}{LE^*(r>Ct)[0,T]}
		\ . \label{ext_LE}
\end{equation}
Note that this estimate only uses the simple bound $\lp{F}{LE(r>Ct)[1,T]} \lesssim 
\sup_{1\leqslant t\leqslant T} \lp{F|_{\mathcal{C}_t}}{L^2(dV_{g|_{\mathcal{C}_t}})}$
where $\mathcal{C}_t$ is the family of hypersurfaces given by $r=Ct$.

\begin{figure}[!ht]
	\scalebox{.75}{\includegraphics{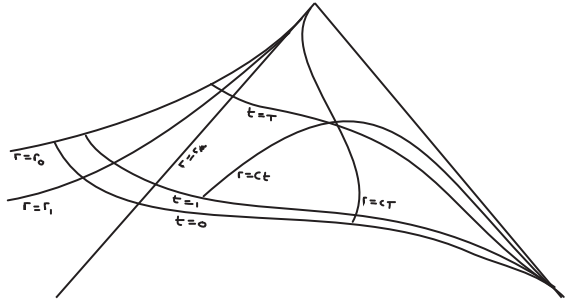}}
	\caption{Penrose diagram of the regions described in the proof.}\label{penrose}
\end{figure}

It remains to estimate $F$ in the region $r<Ct$ which  we enlarge to the region $r<CT$.
After truncating $F_{\alpha\beta}$  by a smooth cutoff of the form $\chi(r/T)$ 
 we  can immediately reduce matters to proving \eqref{main_thm_est_red}
for fields  supported in $r<2CT$, as the error generated
by differentiation of the cutoff $\chi$ on the RHS of \eqref{main_thm_est_red}
is handled by $T^{-1}\lp{\chi'(r/T)F}{LE^*[0,T]}\lesssim \lp{F}{LE(r>Ct)[0,T]}$ 
and then using \eqref{ext_LE}.

To proceed further we average estimates \eqref{div_curl_est} and \eqref{spin_zero_LE}
on the slabs $[0,T]\times \{r>r_*\}$ for $r_*$ in the band $r_0<r_*<r_1<r_M$. The purpose of this averaging
is  simply to trade the resulting error terms $\lp{\big(\star_h I_0 (r_*),\star_h J_0 (r_*)\big)}
{L^2(dV_{\mathbb{S}^2} dt )[0,T]}$ on RHS \eqref{spin_zero_LE} for 
$\lp{(I,J)}{LE^*(r>r_0)[0,T]}$. Combining 
the result with estimate
\eqref{maxwell_energy} yields (again for $F$ supported where $r<2CT$):
\begin{equation}
		 \lp{F(r_1)}{L^2(dV_{g|_{r=r_1}})} +
		 \lp{(w_{ln})^{-1} F}{LE(r>r_1)[0,T]} \ \lesssim \ \lp{F(0)}{L^2(dV_{g|_{t=0}})} 
		+ \lp{w_{ln}(I,J)}{LE^*(r>r_0)[0,T]}
		\ . \notag
\end{equation}
Finally, to fill in the remaining quantity $\lp{(w_{ln})^{-1} F}{LE(r_0<r<r_1)[0,T]}$
it suffices to integrate the LHS of estimate \eqref{maxwell_energy} 
over the family of uniformly space-like hypersurfaces
$r=r_*$ for $r_*$ in the band  $r_0\leqslant r_*\leqslant r_1$.

\ret

\section{Local Energy Decay for Spin-Zero Fields Part I: A Preliminary Estimate}\label{LE_section1}

In this section we begin our estimates for the spin-zero wave equation
\eqref{spin_0_box}.  It is best to think of this as a $(1+1)$ wave
equation plus a potential, which is literally true after decomposition
into spherical harmonics. There is a natural volume form $dV=dV_h \wedge
dV_{\mathbb{S}^2}$ for which $\Box^0$ becomes self adjoint.  Here
$dV_h=\sqrt{|h|}dx^0 \wedge dx^1$ and $dV_{\mathbb{S}^2}$ is the
standard volume on $\mathbb{S}^2$. There is also a natural (but not
exactly conserved) energy type norm:
\begin{equation}
		E( \phi[t]) \ = \ \int_{[r_0,\infty]\times\mathbb{S}^2 } 
		( \phi_t^2 + \phi_r^2 + r^{-2}|\sd \phi(t)|^2) dr dV_{\mathbb{S}^2} \
		, \qquad |\sd \phi|^2 \ = \ g^{AB}\snabla_A\phi\snabla_B\phi  \ , \notag
\end{equation}
where $\snabla$ is the gradient on $\mathbb{S}^2$. Here we take $(t,r)$ to be any regular set
of coordinates as in Definition \ref{black_hole_defn}.
We also set up local energy decay norms as follows:
\begin{equation}
		\lp{\phi}{LE_0} \ = \ \sup_{j\geqslant 0} 2^{-\frac{j}{2}} \lp{\chi_j \phi}{L^2(dV)} \ , 
		\qquad \lp{G}{LE^*_0} \ = \ \sum_{j\geqslant 0} 2^{\frac{j}{2}} \lp{\chi_j G}{L^2(dV)}
		\ \notag ,
\end{equation}
for some overlapping set dyadic cutoffs $\chi_j(r)$. Here   the integrals
are implicitly constrained to the 
region  $r\geqslant r_0$ which we assume is covered by $j\geqslant 0$. 
The main result of this section can now be written as:

\begin{theorem}[Local energy decay for $\Box^0$]\label{Main_LE_prop}
Let $\Box^0\phi=G$, then one has:
\begin{equation}
		\lp{\partial_r\phi }{LE_0[0,T]}
		+ \lp{(w_{ln})^{-1}(\partial_{t}\phi , r^{-1}\sd \phi)}{LE_0[0,T]}  \ \lesssim \
		E^\frac{1}{2}(\phi[0]) + \lp{w_{ln}G}{LE^*_0[0,T]} \ , \label{box0_LE}
\end{equation}
where $w_{ln}(r)=(1+ \big|\ln|r-r_\mathcal{T}|\big|)/(1+|\ln(r)|)$, 
and where  we assume $\la \partial_r , \partial_t\ra_h  =0$ at  $r=r_\mathcal{T}$. 
\end{theorem}

\begin{remark}\label{duhamel_remark}
  By Duhamel's principle (note $g^{00}\approx -1$ in regular
  coordinates) one may replace the norm for $G$ on the RHS of
  \eqref{box0_LE} with the following:
\begin{equation}
		\llp{G}{N_0[0,T]} \ = \ \inf_{G_1+G_2=G} \big( \lp{w_{ln}G_1}{LE^*_0[0,T]}
		 + \lp{G_2}{L^1(L^2)[0,T]}\big) \ , \label{split_norm}
\end{equation}
where $L^1(L^2)$ is with respect to $dt$ and $drdV_{\mathbb{S}^2}$ as above.
\end{remark}

\begin{remark}
  The estimate \eqref{box0_LE} is true for a general $\phi$ even if it
  is spherically symmetric.  On the other hand, if
  $\int_{\mathbb{S}^2}\phi(t,r) dV_{\mathbb{S}^2} \equiv 0$ then by
  Poincare's estimate one can also add an LE estimate for
  $r^{-1}\phi$.  However, a bound for $r^{-1}\phi$ in the spherical
  symmetric case is impossible due to constant solutions which have
  finite energy even after truncation outside a large dyadic set. Note
  that this is in stark contrast to the $LE$ estimates for the scalar
  wave equation $\Box_g$ for which one can estimate $r^{-1}\phi$ even
  in the spherical case (see \cite{MMTT_sch}).
\end{remark}

\begin{remark}
The condition that $\langle \partial_r,\partial_t\rangle=0$ 
at $r=r_\mathcal{T}$
is essential if one wishes to gain an unweighted estimate for $\partial_r\phi$. 
There is an alternate 
geometric description of this requirement which is that
the unweighted derivative must correspond to  
quantizations of the defining functions for the stable and unstable manifolds
of the trapped set.  
\end{remark}

The remainder of this section will be devoted to the proof of Theorem \ref{Main_LE_prop}.
The method we  employ  is to estimate separately the contributions 
coming from the  region $\la dr, dr\ra\approx 0$ (event horizon), and 
the  region $\la dr, dr\ra>0$ (domain of outer communication). 
This allows us to build estimates in a modular way, which we believe should 
also be useful for further applications. 

Close to $r=r_M$ our estimates are purely multiplier based,
and boil down to the  construction of well chosen null frames. There are two key
estimates here: \eqref{red_shift_LE} and \eqref{dt_LE} below. The first
captures the contribution of waves parallel 
to $r=r_M$, and is simply a version of the ``red shift'' estimates first introduced
in \cite{DR_sch} (see also \cite{DR_survey} for a more general exposition). The second
bounds the contribution of waves transverse to $r=r_M$, and turns out to be  a version 
of the conservation of energy.

The region $r>r_M$ involves the most work, and is handled by estimate
\eqref{exterior_LE} below. To prove it we introduce a generalized
``Regge-Wheeler'' type coordinate condition to put the radial part
$h_{ab}$ of the metric in conformal form.  After a little bit of work
to show that the truncated estimate \eqref{exterior_LE} follows from
an untruncated version in Regge-Wheeler coordinates, the analysis
follows along the lines of \cite{BS_uniform} and \cite{MMTT_sch} which
is essentially microlocal in nature.


\subsection{Description of the geometry}

In this section we compute some special coordinate systems in the regions where
$\la dr, dr\ra \approx 0$ and where
$\la dr, dr\ra >0$. 

First, recall  that according to part ii) of Definition 
\ref{black_hole_defn} there exists  a unique $r_M>r_0$ with the property
 that $g^{rr}=\la dr, dr\ra|_{r=r_M}=0$.
Furthermore, the assumption that $\partial_r g^{rr}\neq 0$ throughout $r\geqslant r_0$ 
implies $\partial_r g^{rr}|_{r=r_M}>0$. This is the key non-degeneracy condition which 
leads to good estimates in the region $\la dr, dr\ra \approx 0$. To capture
it we make the following definition:

\begin{deff}[Negatively N-boosted null pairs]\label{null_pair_defn}
Let $\la\cdot, \cdot \ra$ denote the $(1+1)$ Lorenztian inner product of $h$.
For a fixed number $N>0$ we call  a pair of vector-fields $L,\bL$, defined  over a time independent 
neighborhood $\mathcal{H}$ of $r=r_M$, a  
``negatively N-boosted null pair''  if there exists a fixed $c=c(N)>0$ such that the following hold:
\begin{enumerate}[i)]
		\item (Basic relations)  One has the relations $\la L, L\ra=\la\bL, \bL\ra =0 $ and $\la L, \bL\ra=-2$.
		Both $L$ and $\bL$ are future directed in the sense that
		$L t,\bL t>0$.
		In addition $\bL$ is incoming in the sense that $\bL r<-c<0$.
		\item (Stationarity) The frame is stationary,  that is $[\partial_t , L]=[\partial_t , \bL]=0$.
		\item (Boosting condition)
		One has $\nabla_{\bL}\bL= 2N \underline{\chi} \bL$, where $N>0$ is  as above and
		$\underline{\chi}=\bL r$.
\end{enumerate}
\end{deff}

\begin{remark}
Condition iii) above is malleable in the sense that one can construct
pairs of vector-fields $L,\bL$ which satisfy i) and ii), and which have 
$\nabla_{\bL}\bL=\underline{\sigma}\, \bL$, where $\underline{\sigma}$ is
an arbitrary function of $r$. Our choice of constant in iii) is for later convenience when
constructing multipliers.
\end{remark}

In terms of such null pairs close to $r=r_M$ we have the following result, which follows closely 
the presentation of Dafermos-Rodnianski in \cite{DR_survey}:

\begin{lemma}[Description of the geometry close to $r=r_M$]\label{horizon_lemma}
Fix a number  $N>0$. Then under the assumptions of Definition 
\ref{black_hole_defn} there exists a (time independent) 
neighborhood $\mathcal{H}$ of $r=r_M$ and a 
 negatively N-boosted null pair  defined over $\mathcal{H}$ as in Definition \ref{null_pair_defn} above.
Furthermore, for any such null pair    there exists a constant $c=c(N)>0$ such that:
\begin{enumerate}[i)]
		\item (Relation to $\partial_t$) One may write $\partial_t = q_+ L + q_-\bL$ where $q_+>c$, and where
		$q_-=\lambda g^{rr}$ for some weight function $\lambda(r)$ with  $\lambda>c$
		as well.\label{dt_rel}
		\item (Red Shift) One has the identity  $\nabla_L L = \sigma L$ for some smooth function 
		$\sigma > c$.
		\item (Area variations) Define $\underline{\chi}=\bL r$ and $\chi=Lr$. Then in addition
		to $\underline{\chi}<-c$, one also has $\chi=\gamma g^{rr}$ for some smooth function 
		$\gamma>c$.\label{area_var}
\end{enumerate} 
\end{lemma}

It remains to discuss the geometry in the region $r>r_M$. For this we have:

\begin{lemma}[Conformal coordinates in  $r>r_M$]\label{exterior_lemma}
Let $h_{ab}=g_{ab}$ denote the radial part of the Lorentzian metric from 
Definition \ref{black_hole_defn}. For the statement of this lemma we also set
$|h|=|det(h)|$, where the determinant is computed in regular $(t,r)$ coordinates.
Then  there
exist two functions $s=t + b(r)$ and $r_*=r_*(r)$ defined in the region $r>r_M$, such that
$r_*(r_\mathcal{T})=0$, and such that $(s,r_*)$  are smooth
coordinates in $r>r_M$ with:
\begin{equation}
		r_* \ \to \ -\infty \hbox{\ \ as\ \ } r\ \to \ r_M \ , 
		\qquad\qquad 
		r_* \ \to \ \infty \hbox{\ \ as\ \ } r\ \to \ \infty \ . \notag
\end{equation} 
Furthermore the following properties hold:
\begin{enumerate}[i)]
		\item (Asymptotics in $r_M < r < C$) One has (note $g^{rr}=h^{rr}$):
		\begin{equation}
				\partial_ r s \ = \  -|h |^{-\frac{1}{2}} (g^{rr})^{-1} + \td{s}(r) \ , \qquad
				\partial_r r_* \ = \ | h|^{-\frac{1}{2}} (g^{rr})^{-1}
				\ , \notag
		\end{equation}
		where  $\td{s}$ is uniformly bounded
		with all of its derivatives on $r>r_M$. Here all derivatives are taken with respect to regular $(t,r)$
		coordinates as in Definition \ref{black_hole_defn}.
		\item (Asymptotics as $r\to \infty$) As $r\to\infty$ one  has the asymptotics:
		\begin{equation}
				\partial_ r^k s \ = \   O(r^{-k}) \ , \qquad k\geqslant 1 \ , \notag
		\end{equation}
		as well as:
		\begin{equation}
				\partial_r r_* \ = \ 1+ O(r^{-1}) \ , \qquad\qquad
				\partial_r^k r_* \ = \ O(r^{-k}) \ , \qquad k\geqslant 2
				\ . \notag
		\end{equation}
		Again everything is computed with respect to regular $(t,r)$
		coordinates.
		\item (Conformal form) With respect to $(s,r_*)$ the $(1+1)$ dimensional
		Lorentzian metric $h$ can be written as:
		\begin{equation}
				h \ = \ \Omega^2 (-ds^2+dr_*^2) \ , \qquad
				\Omega^2 \ = \ -g_{00} \ =\  |h | g^{rr} \ . \notag
		\end{equation}
		\item (Trapped set) The unique trapped set of part iii) of Definition \ref{black_hole_defn}
		is at $r_*=0$. At $r_*=0$ one has $V'(0)=0$ and $V''(0)<0$, where $V=V(r_*)$
		is the effective potential $V=r^{-2}\Omega^2$. Moreover, $V'\neq 0$
		for $r_*\neq 0$. 
		\item (Asymptotics of $V(r_*)$) Finally, $V(r_*)$ has the following asymptotic formulas 
		some fixed $c>0$:
		\begin{equation}
				\partial_{r_*}^k V(r_*) \ \approx  \ (r-r_M) \ \approx \ e^{cr_*} \hbox{\ \ as \ \ } r_*\to-\infty \ , 
				\qquad\qquad 
				(-1)^k\partial_{r_*}^k 
				V(r_*) \ \approx \ r^{-2-k} \ \approx \ r_*^{-2-k} \hbox{\ \ as \ \ } r_*\to \infty \ . \notag
		\end{equation}
\end{enumerate}
\end{lemma}\ret


\begin{proof}[Proof of Lemma \ref{horizon_lemma}]
We proceed in a series of steps:

\step{1}{Construction of the  null pair with boosting condition}
First we show the existence of $L,\bL$ according to Definition \ref{null_pair_defn}.
We denote by $(t,r)$ any regular system of coordinates as in Definition 
\ref{black_hole_defn}, and by $h_{ab}=g_{ab}$ the components
of the radial part of $g$. 
Set $T=-\nabla t$ the future directed time function gradient. 
Then we have:
\begin{equation}
		\la T,T\ra \ = \ h^{00} \ < \ 0 \ , \qquad \la T,\partial_r \ra \ = \ 0 \ , \qquad
		\la \partial_r , \partial_r \ra \ = \ h_{rr} \ > \ 0 \ . \notag 
\end{equation}
From these one can form a pair of null vectors:
\begin{equation}
		\ell \ = \ T + \sqrt{-h^{00}/h_{rr}}\partial_r \ , \qquad
		\underline{\ell} \ = \ T- \sqrt{-h^{00}/h_{rr}}\partial_r \ . \notag
\end{equation}
Immediately we have $\ell t=\underline{\ell} t=-h^{00}>0$. In addition note that $\la\ell,\partial_r\ra>0$,
thus $\ell$ is the outgoing null direction and in particular $\ell r|_{r=r_M}=0$. This also shows 
$Tr|_{r=r_M}<0$.
Choosing $L=\lambda \ell$ and $\bL=\underline{\ell}$ for an appropriate positive weight function
$\lambda(r)$ we immediately have the inner product and future direction
properties and of part i) of Definition \ref{null_pair_defn}, as well as
the stationarity condition ii). Furthermore, one computes $\underline{L} r =  Tr - \sqrt{-h^{00}/h_{rr}}<0$
where the inequality holds in a neighborhood of $r=r_M$.

To establish part iii) of Definition \ref{null_pair_defn} we remark that any null pair obeying 
parts i)-ii) is invariant with respect to boosting $L\to q^{-1} L$ and $\bL\to q \bL$
for a smooth positive weight function $q$. There exists a function $\underline{\sigma}$ such that
$\nabla_{\bL}\bL=\underline{\sigma}\, \bL$. Set $\td{\bL}=q \bL$ and define $\td{\underline{\sigma}}$
accordingly. Then a quick computation shows that:
\begin{equation}
		\td{\underline{\sigma}} \ = \ \underline{\chi}q' + \underline{\sigma}q \ , 
		\qquad\qquad \hbox{where \ \ \ \ } \underline{\chi} \ = \ \bL r \ . \notag
\end{equation}
Setting $\td{\underline{\sigma}}=2N \td{\bL} r=2Nq \underline{\chi}$ gives the following ODE for
$q(r)$:
\begin{equation}
		q' \ = \ \big(2N-(\underline{\chi})^{-1}\underline{\sigma}\big)q \ , \notag
\end{equation}
which is well defined on account of the condition $\underline{\chi}\neq 0$. Such an equation can always
be solved on $\mathcal{H}$ subject to the constraint $q >0$.

\step{2}{Relation to $\partial_t$}
To prove part i) of Lemma \ref{horizon_lemma} we set $q_\pm$ to 
be the coefficients of $\partial_t$ in the basis $\{L,\bL\}$, and note that
$q_\pm>0$ where $\partial_t$ is time-like. Then one has:
\begin{equation}
		q_+q_- \ = \ -\frac{1}{4}h_{00} \ = \ -\frac{1}{4}\det(h) g^{rr} \ . \label{qq_iden}
\end{equation}
By assumption \ref{r_M_assumption}) of Definition \ref{black_hole_defn}
the RHS is a function with simple zero at $r=r_M$. Thus only one of $q_{\pm}$ can vanish, and 
it does so with a simple zero at $r=r_M$. But $r=r_M$ is a null hypersurface with generator proportional
to $L$   and on this hypersurface $\partial_t$ is also null (see part b) of
Remark \ref{BH_def_rem}). Thus $\partial_t |_{r=r_M}=cL|_{r=r_M}$ for some $c>0$, 
which shows that $q_-$ must have  a simple zero at $r=r_M$ and $q_+>0$ throughout
$\mathcal{H}$. Referring back to \eqref{qq_iden} gives the formula for $q_-$ in terms of 
$g^{rr}$ and a non-vanishing weight.

\step{3}{The ``red shift''}
Using the results of the previous steps we have:
\begin{equation}
		\sigma \ = \ -\frac{1}{2}\la\nabla_L L , \bL\ra \ =\  \frac{1}{2}\la \nabla_L \bL , L\ra
		\ = \ \frac{1}{2}q_+^2 \la \nabla_{\partial_t} \bL , \partial_t \ra + O(r-r_M)
		\ = \ \frac{1}{4}q_+^2 \bL h_{00} + O(r-r_M) \ . \notag
\end{equation}
For the first term on the RHS we further compute:
\begin{equation}
		\bL h_{00} \ = \ \bL(r)\partial_r \big( \det(h)^{-1}h^{rr}\big) \ = \ \bL(r) \det(h)^{-1} \partial_r g^{rr}
		+ O(r-r_M) \ > \ 0 \ , \qquad \hbox{if\ \ } |r-r_M|\ll 1 \ . \notag
\end{equation}

\step{4}{Area variation}
In the construction of $L,\bL$ we have already shown that $\bL r<-c$. Using \textbf{Step 2} above
we compute:
\begin{equation}
		0 \ = \ \partial_t r \ = \ q_+ L r + \lambda g^{rr} \bL r \ . \notag
\end{equation}
The desired result follows by solving for $Lr$ and using the already established 
properties of $\lambda,q_+,\bL r$.
\end{proof}\ret

Next, we construct conformal coordinates in the region $r>r_M$.

\begin{proof}[Proof of Lemma \ref{exterior_lemma}]
Recall that our $(1+1)$ metric $h$ as the form $h=g_{ab}dx^adx^b$ with the 
conditions of Definition \ref{black_hole_defn}. We let $g_{00},g_{0r},g_{rr}$ denote 
the components of $h$ in the original (non-singular) $(t,r)$ coordinates. As in the statement 
of Lemma \ref{exterior_lemma} we also set $|h|=|\hbox{det}(h)|$ computed with respect to 
$(t,r)$.

\step{1}{Normalized coordinates} 
To normalize things in the region $g^{rr}>0$
we first introduce a singular time function whose level sets are perpendicular to the Killing field
$\partial_t$. Setting  $t=s-b(r)$ we have in the $(s,r)$ coordinates:
\begin{equation}
		h \ = \ g_{00} ds^2 + 2(g_{0r}-g_{00} b')dsdr + (g_{rr} + g_{00}(b')^2 - 2g_{0r}b')dr^2 \ .  \notag
\end{equation}
Now choose $b(r)$ so that $g_{00} b' = g_{0r}$ and $b(r_\mathcal{T})=0$, 
so our metric takes the diagonal form:
\begin{equation}
		h \ = \ h_{ss}ds^2 + h_{rr}dr^2 \ , \label{sch_form}
\end{equation}
where:
\begin{equation}
		h_{ss} \ = \ g_{00} \ , \qquad
		h_{rr} \ = \ g_{rr} - (g_{0r})^2/g_{00} \ , \label{h_components}
\end{equation}
which is now only defined in the region $g_{00}<0 \ \Leftrightarrow\
r>r_M$. We remark that this change of coordinates
does not affect the determinant of the metic, and in either case we have:
\begin{equation}
		|h| \ = \ - h_{ss}h_{rr} \ = \ (g_{0r})^2-g_{00}g_{rr} \  . \label{det_form}
\end{equation}
We refer to \eqref{sch_form} as the ``Schwarzschild form'' of the metric $h$.

Next, we make an additional transformation to put the metric in conformal form by defining 
$r_*=r_*(r)$ according to the ODE:
\begin{equation}
 		\frac{dr_*}{dr} \ = \ \sqrt{-h_{rr}/h_{ss}} \ , \qquad  r_*(r_\mathcal{T}) \ = \ 0 \ . \label{r_*_equation}
\end{equation}
Using  the identity $h_{ss}=-| h| g^{rr}$ and the fact that $h_{rr}=1/g^{rr}$
the first identity on this last line becomes $\partial_r r_* = | h |^{-\frac{1}{2}} (g^{rr})^{-1}$,
which gives the second formula in part i) of the Lemma and immediately yields the rough asymptotics:
\begin{equation}
		\frac{dr_*}{dr} \ \approx \ 1/g^{rr} \ \approx \ 1/(r-r_M)\ , \quad r\to r_M \ , \qquad\qquad
		\frac{dr_*}{dr} \ \approx \ 1 \ , \quad r\to \infty \ . \notag
\end{equation}
Thus $r_*\to-\infty$ as $r\to r_M$ and $r_*\to \infty$ as $r\to\infty$, and our metric takes the
final form:
\begin{equation}
		h \ = \ \Omega^{2} ( -ds^2 + dr_*^2) \ , \qquad \Omega^{2}  \ = \ -\la \partial_t , \partial_t \ra_g \ ,
		\qquad\qquad -\infty < r_* < \infty \Longleftrightarrow r>r_M \ . \label{RW_form}
\end{equation}
We call this the ``Regge-Wheeler form'' of $h$.

\step{2}{Asymptotics of $s$ as $r\to r_M$}
Here we derive the first formula in i) of the Lemma. From the formulas of the previous step
we have:
\begin{equation}
		\partial_r s \ = \ {g_{0r}}/{g_{00}} \ , \qquad\qquad
		\hbox{where\ \ \ \ }
		g_{00} \ = \ -|h|  g^{rr} \ , \qquad
		\hbox{and\ \ \ \ }g_{0r} \ = \ \sqrt{|h| + g_{00}g_{rr}} \ . \notag
\end{equation}
This allows us to write:
\begin{equation}
		\partial_r s \ = \ -|h |^{-\frac{1}{2}}(g^{rr})^{-1}\cdot \big(1 - f(r,g^{rr})\big) \ , 
		\qquad\qquad f(r,\theta) \ = \ 1- \sqrt{1 - \theta g_{rr}} \  . \notag
\end{equation}
Thus
$f(r,\theta)=\theta \td{f}(r,\theta)$ when $\theta\approx 0$, 
for some other smooth  $\td{f}$, and the desired result follows.

\step{3}{Asymptotics of $(s,r_*)$ as $r\to \infty$}
The result for $s$ follows immediately from $\partial_r^k (g_{ab})=O(r^{-k})$ and 
$(g_{ab}-\eta_{ab})=O(r^{-1})$, where $\eta=\diag(-1,1)$.
The result for $r_*$  also follows from these asymptotics for $g_{ab}$ and the 
formulas on lines \eqref{h_components} and \eqref{r_*_equation}.

\step{4}{Description of the trapped set at $r=r_\mathcal{T}$} 
The Regge-Wheeler form \eqref{RW_form} is particularly convenient for discussing the 
trapped null geodesics of $g$ in the region $r>r_M$. Since conformal metrics have the same
null geodesic flow, it suffices to analyze the direct sum metric:
\begin{equation}
		\td{g} \ = \ -ds^2 + dr_*^2 + \Omega^{-2}r^2 d\omega^2 \ , \qquad 
		d\omega^2 \ = \ \delta_{AB}dx^Adx^B \ . \notag
\end{equation}
The projections of unit speed 
null   geodesics for this metric onto the space parametrized by
$(r_*,x^A)$, where $A=2,3$ denote coordinates on $\mathbb{S}^2$,    
are exactly the unit speed geodesics of the three dimensional
Riemannian  surface of rotation $dl^2=dr_*^2 + \Omega^{-2}r^2 d\omega^2$.
The Hamiltonian for the corresponding geodesic  flow is $p(r_*,x^A,\xi_{r_*},\xi_A)=\frac{1}{2}(\xi_{r_*}^2 
+ \Omega^{2}r^{-2}|\sxi|^2)$, where
$|\sxi|^2=\delta^{AB}\xi_A\xi_B$ is the Hamiltonian of the standard
geodesic flow on $\mathbb{S}^2$.

The equations for a trapped sphere of unit speed geodesics at $r=r_0$ now take the form:
\begin{equation}
		\dot r_*|_{r=r_0} \ = \ \{ p , r_*\}|_{r=r_0}\ = \ 0 \ , \notag
\end{equation}
where $\{p,f\}=\partial_\xi p\partial_x f - \partial_x p\partial_\xi f$ is the Poisson bracket.
In general we have:
\begin{equation}
		\dot r_* \ = \ \xi_{r_*} \ , \qquad \dot \xi_{r_*} \ = \ -\frac{1}{2}\partial_{r_*} (V)|\sxi|^2 \ , \qquad\qquad
		\hbox{where\ \ \ \ } V \ = \ \Omega^{2} r^{-2}  \ , \notag
\end{equation}
and thus $r=r_{0}$ is trapped iff $\xi_{r_*}=\dot\xi_{r_*}=0$ at $r=r_0$, which can happen iff 
$V'(r_0)=0$.
Therefore, condition iii) of Definition \ref{black_hole_defn} implies that $\partial_{r_*}V=0$ iff $r_*=0$
in the region $-\infty < r_* < \infty$. It is also manifestly clear that $V>0$ in that region with  
$V\to 0$ as $r_*\to\pm \infty$. In other words $V(r_*)$ is a positive repulsive potential with
unique maximum at $r_*=0$.\footnote{This also immediately implies there are no other trapped 
geodesics except at $r=r_{\mathcal{T}}$.} 
Finally, the non-degeneracy  condition in iii) of 
Definition \ref{black_hole_defn} means that $V''(0)\neq 0$, so in fact 
$V''(0)< 0$.
Summarizing this, we may write the normal part of the geodesic 
flow at $r=r_\mathcal{T}$ in the form of a standard planar hyperbolic\footnote{In particular the 
assumption of a unique non-degenerate trapped set
implies such a trapped set must be normally hyperbolic.}  fixed point:
\begin{equation}
		\dot r_* \ =  \ \xi_r \ , \qquad
		\dot \xi_r \ = \ \lambda_0 |\sxi|^2 r_*  + O(|\sxi|^2r_*^2 ) \ , \qquad
		\lambda_0 \ =\  -\frac{1}{2} V''(0) >0 \ . \notag
\end{equation}

\step{5}{Asymptotics of $V$} 
Finally  the asymptotic formulas of part v) for $V(r_*)$ as $r_*\to\pm\infty$ follows from the explicit 
formula above and the fact that $r_*=A\ln (r-r_M)+O(1)$ as $r_*\to-\infty$, and $r_*=r+B\ln(r)+O(r^{-1})$
as $r_*\to \infty$, where $A>0$ and $B$ are constant. Further details are left
to the reader.
\end{proof}


\subsection{The multiplier method for $\Box^0$}


Before embarking on the proof of Theorem \ref{Main_LE_prop} we pause 
to introduce some notation and identities that will be useful for the remainder of the section.
For the wave equation $\Box^0$ one has a $(1+1)$ energy momentum tensor:
\begin{equation}
		Q_{ab}[\phi] \ = \ \int_{\mathbb{S}^2}\big[\partial_a \phi\partial_b\phi 
		- \frac{1}{2}h_{ab}(\partial^c\phi \partial_c\phi + r^{-2}|\sd\phi|^2 )\big]dV_{\mathbb{S}^2}
		\ . \notag
\end{equation}
In terms of any null pair basis $L,\bL$ one can write its components as:
\begin{equation}
		Q_{LL}[\phi]  \ = \  \int_{\mathbb{S}^2}(L\phi)^2dV_{\mathbb{S}^2} \ , \qquad 
		Q_{\bL\, \bL}[\phi]  \ = \  \int_{\mathbb{S}^2}(\bL\phi)^2dV_{\mathbb{S}^2} \ , \qquad
		Q_{L\bL}[\phi]  \ = \  \int_{\mathbb{S}^2} |\sd\phi|^2dV_{\mathbb{S}^2} \ . \label{Q_components}
\end{equation}
This object behaves similarly to the energy momentum tensor for $\Box_h + V$ where $V=-k^2r^{-2}$.
Specifically one has:
\begin{equation}
		\nabla^a Q_{ab}[\phi] \ =\  
		 \int_{\mathbb{S}^2} \Box^0\phi \partial_b\phi\,  dV_{\mathbb{S}^2} 
		 + r^{-3}\partial_b(r) \int_{\mathbb{S}^2} |\sd\phi|^2dV_{\mathbb{S}^2} \ . \notag
\end{equation}
In the usual way this leads to momentum identities for any (spherically 
symmetric) vector-field $X=X^a\partial_a$:
\begin{equation}
		\nabla^a {}^{(X)}\! P_a[\phi]  =  \frac{1}{2}
		\int_{\mathbb{S}^2}\!\!\! Q_{ab}[\phi]{}^{(X)}\! \pi^{ab}  dV_{\mathbb{S}^2} 
		 \!+\! r^{-3}X(r) \int_{\mathbb{S}^2}\!\!\! |\sd\phi|^2dV_{\mathbb{S}^2} 
		 \!+\! \int_{\mathbb{S}^2}\!\!\! \Box^0\phi X\phi\,  dV_{\mathbb{S}^2} \ , \ \
		  {}^{(X)}\! P_a[\phi]  =  Q_{ab}[\phi] X^b \ , \label{momentum}
\end{equation}
where ${}^{(X)}\! \pi_{ab}=\mathcal{L}_X h_{ab}= \nabla_a X_b + \nabla_b X_a$ is the 
$(1+1)$ deformation tensor of $X$.

There are two types of multipliers we will work with:
\begin{equation}
		X \ = \ X^- \bL \ , \qquad 
		Y \  =\ Y^0 \partial_t \ , \label{multipliers}
\end{equation}
for given weight functions $X^-(r),Y^0(r)$. We record their raised 
deformation tensors here:

\begin{lemma}[Deformation tensors]\label{XY_def_lemma}
Let $X,Y$ be the multipliers defined on line \eqref{multipliers}.
 Then with the notation of 
Definition \ref{null_pair_defn} and Lemma \ref{horizon_lemma} we have:
\begin{align}
		{}^{(X)}\! \pi^{LL} \ &= \ 0 \ , 
		&{}^{(X)}\! \pi^{\bL\, \bL} \ &= \ - \chi \cdot  (X^-)' + {\sigma}X^- 
		&{}^{(X)}\! \pi^{L\bL} \ &= \ -\frac{1}{2}\bchi\cdot(X^-)' - N\underline{\chi} X^-
		\label{X_def}\\
		{}^{(Y)}\! \pi^{LL} \ &= \ -\bchi \cdot  q_+ (Y^0)' \ ,
		&{}^{(Y)}\! \pi^{\bL\, \bL} \ &= \ - \chi \cdot  q_- (Y^0)' \ ,
		  &{}^{(Y)}\! \pi^{L\bL} \ &= \ 0 \ . \label{Y_def}
\end{align}
and where we have set $\chi=Lr$ and $\bchi=\bL r$. Close to $r=r_M$ one has the sign relations:
\begin{equation}
		\chi \ \approx \ (r-r_M) \ , \qquad
		\bchi \ \approx \ -1 \ , \qquad \sigma \ \approx \ 1 \ , \qquad 
		q_+ \ \approx \ 1 \ , \qquad
		q_- \ \approx \ (r-r_M)
		\ . \label{signs}
\end{equation}
\end{lemma}

\begin{proof}[Proof of Lemma \ref{XY_def_lemma}]
By raising indices we have ${}^{(X)}\! \pi^{LL}=\frac{1}{4}{}^{(X)}\! \pi_{\bL\, \bL}$, 
${}^{(X)}\! \pi^{\bL\, \bL}=\frac{1}{4}{}^{(X)}\! \pi_{L L}$, and ${}^{(X)}\! \pi^{L\bL}=\frac{1}{4}{}^{(X)}\! \pi_{L\bL}$.
Similarly for ${}^{(Y)}\! \pi^{ab}$. For $X=X^-\bL$ we have:
\begin{equation}
		{}^{(X)}\! \pi_{ab} \ = \ e_a(X^-) \la \bL, e_b\ra + e_b(X^-) \la \bL, e_a\ra
		+ X^- \la \nabla_{e_a}\bL, e_b\ra  +X^- \la \nabla_{e_b}\bL, e_a\ra \ . \notag
\end{equation}
Using the orthogonality of $\bL$ to itself we immediately have ${}^{(X)}\! \pi_{\bL\, \bL}=0$, while:
\begin{equation}
		{}^{(X)}\! \pi_{L L} \ = \ -4 L(X^-) - 2X^-\la\nabla_{L} L, \bL\ra \ , 
		\qquad {}^{(X)}\! \pi_{L \bL} \ = \ -2 \bL(X^-) + X^-\la \nabla_{\bL} \bL , L\ra \ . \label{-_calcs}
\end{equation}
This immediately implies \eqref{X_def}.

Now let $Y=Y^0\partial_t$, for which we have:
\begin{equation}
		{}^{(Y)}\! \pi_{ab} \ = \ e_a(Y^0) \la \partial_t, e_b\ra + e_b(Y^0) \la \partial_t, e_a\ra \ , \notag
\end{equation}
where we have used ${}^{(\partial_t )}\! \pi_{ab}=\la\nabla_a \partial_t, e_b\ra 
+ \la\nabla_b \partial_t , e_a\ra=0$. Using the decomposition $\partial_t = q_+L+q_-\bL$
this immediately gives the first two identities on 
\eqref{Y_def}. For the last term on that line we have:
\begin{equation}
		L(Y^0) \la \partial_t, \bL \ra + \bL(Y^0) \la \partial_t, L\ra \ = \ 
		-2\big(q_+L(Y^0)  + q_-\bL(Y^0) \big) \ = \ -2\partial_t Y^0 \ = \ 0 \ . \notag
\end{equation}

\end{proof}


\subsection{Local energy decay component estimates for $\Box^0$. Proof of Theorem \ref{Main_LE_prop}}

We break our proof of Theorem \ref{Main_LE_prop} down into the
three constituents in the next proposition. For this purpose we consider 
thresholds 
\[
 r_0\leqslant r_1 < r_M < r_2<r_3 < r_\mathcal{T}
\]
We think of $r_3$  as fixed, away from $r_M$ and $ r_\mathcal{T}$.
On the other hand $r_1$ and $r_2$ are free to float, and we reserve the right 
to choose them arbitrarily close to $r_M$. For the statement of our main technical 
result we let $\chi_{[a,b]}(r)$ denote the indicator function (sharp cutoff) 
of the interval $a\leqslant r\leqslant b$.

\begin{proposition}[Three modular local energy decay estimates]\label{modular_prop}
Let $r_1,r_2,r_3$ be three parameters such that $r_0\leqslant r_1 < r_M < r_2<r_3<r_\mathcal{T}$. Then 
there exists a fixed value of $r_3$ such that for every $k>1$  
one has the following three estimates for test functions $\phi$,
where the implicit constants are uniform in $r_1,r_2$ (but may depend on $r_3,k$):
\begin{align}
		&\lp{\chi_{[r_1,r_3]} (\bL\phi,\sd\phi)}{LE_0[0,T]}   \lesssim 
		\lp{\chi_{[r_3,r_\mathcal{T}]} (r-r_\mathcal{T})\sd\phi}{LE_0[0,T]}
		+ E^\frac{1}{2}_{\geqslant r_1}(\phi[0]) + \lp{\chi_{[r_1,r_\mathcal{T}]}\Box^0\phi }{LE^*_0[0,T]}
		 \ , \label{red_shift_LE}\\
		&\lp{\chi_{[r_1,r_2]}w_k L\phi }{LE_0[0,T]}  \lesssim 
		\epsilon \lp{\chi_{[r_1,\infty)}(w_{ln})^{-1}\partial_t\phi}{LE_0[0,T]}+
		\lp{\chi_{[r_1,r_2]}w_k(r-r_M) \bL \phi}{LE_0[0,T]}\label{dt_LE} \\
		&\hspace{.2in} + (r_M-r_1)^\frac{1}{2}
		\lp{\chi_{[r_3,r_\mathcal{T}]} (r-r_\mathcal{T})\sd\phi}{LE_0[0,T]} 
		+ E^\frac{1}{2}_{\geqslant r_1}(\phi[0]) + 
		\epsilon^{-1}\lp{\chi_{[r_1,\infty)}w_{ln}\Box^0\phi}{LE^*_0[0,T]}
		\ . \notag
\end{align}
In addition there exists a $C_{r_2}>0$ such that (again with implicit constant uniform in $r_1,r_2$):
\begin{multline}\label{exterior_LE}
		\lp{\chi_{[r_2,r_3]} 
		w_{k} \big( L\phi, (r-r_M)\bL\phi\big)}{LE_0[0,T]} + 
		\lp{\chi_{[r_2,r_3]} 
		\sd\phi}{LE_0[0,T]}\\
		+ \lp{\chi_{[r_3,\infty)} 
		\partial_r \phi }{LE_0[0,T]}
		+ \lp{\chi_{[r_3,\infty)}(w_{ln})^{-1}(\partial_t\phi,r^{-1}\sd\phi) }{LE_0[0,T]}\\
		 \ \lesssim \ 
		|\ln(r_2-r_M)|^k \lp{\chi_{[r_M,r_2]}w_k \big( L\phi, (r-r_M) \bL\phi)}{LE_0[0,T]}
		+ C_{r_2}\Big( \lp{\chi_{[r_M,\infty)}w_{ln}\Box^0\phi}{LE^*_0[0,T]} + 
		E^\frac{1}{2}_{\geqslant r_M}(\phi[0]) \Big) 
		\ . 
\end{multline}
The weight functions $w_k$ and $w_{ln}$ in estimates \eqref{dt_LE} and \eqref{exterior_LE} are:
\begin{equation}
		w_{ln} \ = \ (1+\big|\ln|r-r_\mathcal{T}|\big|)/(1+|\ln(r)|) \ , \qquad
		w_k(r) \ = \ (1+\big|\ln|r-r_M|\big|)^{-\frac{1}{2}k}|r-r_M|^{-\frac{1}{2}} \ .\notag
\end{equation}
In estimate \eqref{exterior_LE}  one must also 
take $\la\partial_t, \partial_r\ra=0$ when $r=r_\mathcal{T}$.
\end{proposition}

\begin{remark}
The purpose of the weight $w_{ln}$ is to account for the degeneracy of local energy decay
near the trapped set $r=r_\mathcal{T}$. On the other hand the weight $w_k$ is only used
in a neighborhood of the horizon $r=r_M$ and accounts for the fact that the ``good component''
$L\phi$ has a non-degenerate energy there.
\end{remark}

Before giving a proof of the above three estimates, we use them to give a short 
demonstration of Theorem \ref{Main_LE_prop}.

\begin{proof}[Proof that Proposition \ref{modular_prop} implies estimate \eqref{box0_LE}]
First note that while
the $LE_0$ estimate in  \eqref{box0_LE} takes place over the region $r\geqslant r_0$, one can easily
reduce it to showing the same bound for $r\geqslant r_1$, as the region $r_0\leqslant r \leqslant r_1$ is easily 
covered by known $LE_0$ bounds on $r\geqslant r_1$ and energy estimates.
 
 For constants $c,C>0$ we add together:
\begin{equation}
		c\times \eqref{red_shift_LE} + C|\ln(r_2-r_M)|^k\times \eqref{dt_LE} 
		+ \eqref{exterior_LE} \ . \notag
\end{equation}
Then  choose $c,C,\epsilon,r_1,r_2$ so that:
\begin{equation}
		c+C^{-1}\ll 1 \ , \qquad\qquad
		\hbox{followed by\ \ \ \ }
		|r_1-r_2|^\frac{1}{2} + \epsilon \ll cC^{-1}|\ln(r_2-r_M)|^{-k} \ . \notag
\end{equation}
This allows one to move all local energy decay errors for $\phi$ to the LHS of the resulting inequality.  
\end{proof}

We now give a proof of the first two estimates in Proposition \ref{modular_prop}. These are based 
directly on the multiplier calculations of the previous subsection. The last estimate \eqref{exterior_LE}
is  more involved and will be the topic of the next subsection.

\begin{proof}[Proof of estimate \eqref{red_shift_LE}]
For the purposes of showing estimate \eqref{dt_LE} next, it will help to prove a slightly more
informative version of \eqref{red_shift_LE}. Let $L,\bL$ be a negatively N-boosted null pair
(according to Definition \ref{null_pair_defn}), and $r_1,r_3$ sufficiently close to $r_M$ 
so that   Lemma \ref{horizon_lemma} holds on a time independent 
neighborhood $\mathcal{H}$ of $\mathbb{R}_t\times[r_1,r_3]$ with $r=r_\mathcal{T}$ to the exterior
of $\mathcal{H}$.

Next, we choose the 
$X$ multiplier from line \eqref{multipliers} with $X^-(r)=\varphi(r)$ where $\varphi(r)$ is a smooth 
non-negative function with $\varphi\equiv 1$ on $[r_1,r_3]$, $\varphi'\leqslant 0$, and $\varphi\equiv 0$ 
for $r$ outside of $\mathcal{H}$. Denote by $\mathcal{S}$ the space-time slab 
$\{r\geqslant r_1\}\times\{0\leqslant t\leqslant T\}$, and by $\mathcal{S}(t_0)$
the section $\{t=t_0\}\cap \mathcal{S}$ and similarly for $\mathcal{S}(r_0)$.
Then forming the momentum density ${}^{(X)}\!P$ as on line \eqref{momentum}, and integrating over the
$\mathcal{S}$, we have from Stokes' theorem the identity:
\begin{multline}
		\int_{\mathcal{S}(t_0)}\hspace{-.2in} \varphi Q_{T \bL} \sqrt{|h|} dr dV_{\mathbb{S}^2}
		+ \int_{\mathcal{S}(r_1)}\hspace{-.2in} Q_{R\bL}\sqrt{|h|} dV_{\mathbb{S}^2} dt 
		+ \frac{1}{2}\int_\mathcal{S}\hspace{-.05in}
		{}^{(X)}\! \pi^{\bL\, \bL} (\bL\phi)^2\sqrt{|h|}dr dV_{\mathcal{S}^2}dt  \\
		+ \int_\mathcal{S}\hspace{-.05in}\varphi r^{-2}(-\underline{\chi})(N-r^{-1}
		)|\sd\phi|^2 \sqrt{|h|}dr dV_{\mathcal{S}^2}dt 
		\ = \ 
		\int_{\mathcal{S}(0)}\hspace{-.15in} \varphi Q_{T \bL} \sqrt{|h|} dr dV_{\mathbb{S}^2} + 
		\frac{1}{2} \int_\mathcal{S} r^{-2}\underline{\chi}\varphi' |\sd \phi|^2 
		\sqrt{|h|}dr dV_{\mathcal{S}^2}dt \\
		- \int_\mathcal{S} \varphi  \Box^0\phi \bL\phi
		\sqrt{|h|}dr dV_{\mathcal{S}^2}dt \ . \notag
\end{multline}
where we have set
$T  =  -\nabla t $ and  
$R  =  \nabla r  =-\frac{1}{2}\chi\bL -\frac{1}{2} \underline{\chi}L$,  which are the future pointing normals
to $t=0,t_0$ and $r=r_1$. Because the level sets $t=const$ are uniformly space-like we also know that 
$T=T^+L + T^-\bL$ where $T^\pm \approx 1$ on the support of $\varphi$. Thus, for $N>0$ sufficiently
large, the previous 
identity yields the following uniform estimate where the implicit constant does not depend 
$r_1,r_2$:
\begin{multline}
		\lp{\chi_{[r_1,r_3]} (\bL\phi,\sd\phi)|_{t=T}}{L^2(drdV_{\mathbb{S}^2})} +
		\lp{\big((r_M-r_1)^\frac{1}{2}\bL\phi,\sd\phi\big)|_{r=r_1}}{L^2(dV_{\mathbb{S}^2}dt)[0,T]}\\
		+ \lp{\chi_{[r_1,r_3]} (\bL\phi,\sd\phi)}{L^2(drdV_{\mathbb{S}^2}dt)[0,T]}\ \lesssim \ 
		 \lp{\chi_{[r_3,r_\mathcal{T}]}(r-r_\mathcal{T}) \sd\phi }{L^2(drdV_{\mathbb{S}^2}dt)[0,T]}\\
		 + \lp{\chi_{[r_1,r_\mathcal{T}]} (\bL\phi,\sd\phi)|_{t=0}}{L^2(drdV_{\mathbb{S}^2})} 
		 + \lp{\chi_{[r_1,r_\mathcal{T}]}\Box^0\phi}{L^2(drdV_{\mathbb{S}^2}dt)[0,T]}
		 \ . \label{general_red_shift}
\end{multline}
Disregarding the first two terms on the LHS yields the desired result.
\end{proof}

\begin{proof}[Proof of estimate \eqref{dt_LE}]
Here we use the same setup as in the first paragraph of the previous proof. This time we 
use the $Y$ multiplier from line \eqref{multipliers} with $Y^-=q(r)$ where
$q$ is a non-negative function such that  $q'=w_k^2\chi_{[r_1,r_2]}$. 
Thus we have the global bound $C>q>c>0$ for some set of constants $c,C$.
Applying   \eqref{momentum} with ${}^{(Y)}\!P$ yields:
\begin{multline}
		\int_{\mathcal{S}(t_0)}\hspace{-.2in} q Q_{T 0} \sqrt{|h|} dr dV_{\mathbb{S}^2}
		+ \int_{\mathcal{S}(r_1)}\hspace{-.2in} qQ_{R 0}\sqrt{|h|} dV_{\mathbb{S}^2} dt 
		+  \frac{1}{2} \int_\mathcal{S}\hspace{-.05in}(-\underline{\chi}) q_+ w_k^2 \chi_{[r_1,r_2]}
		 (L\phi)^2\sqrt{|h|}dV_{\mathcal{S}^2}drdt  \\
		= \  \frac{1}{2} \int_\mathcal{S}\hspace{-.05in}
		\chi q_-  w_k^2\chi_{[r_1,r_2]}
		|\bL \phi|^2 \sqrt{|h|} dr dV_{\mathcal{S}^2}dt 
		+
		\int_{\mathcal{S}(0)}\hspace{-.15in} q Q_{T 0} \sqrt{|h|} dr dV_{\mathbb{S}^2} - 
		 \int_\mathcal{S} q  \Box^0\phi \partial_t \phi
		\sqrt{|h|}dr dV_{\mathcal{S}^2}dt \ . \notag
\end{multline}
Expanding the boundary terms
using part \ref{dt_rel}) of Lemma \ref{horizon_lemma} we get:
\begin{multline}
		\lp{\chi_{[r_1,r_2]}w_k L\phi}{L^2(dr dV_{\mathbb{S}^2}dt)[0,T]} \ \lesssim \ 
		\lp{\chi_{[r_1,r_2]}w_k(r-r_M)\bL\phi}{L^2(dr dV_{\mathbb{S}^2}dt)[0,T]}\\
		+ (r_M-r_1)^\frac{1}{2} \lp{\chi_{[r_1,r_M]} \bL\phi|_{t=T}}{L^2(drdV_{\mathbb{S}^2})} +
		(r_M-r_1)^\frac{1}{2}\lp{\big((r_M-r_1)^\frac{1}{2}
		\bL\phi,\sd\phi\big)|_{r=r_1}}{L^2(dV_{\mathbb{S}^2}dt)[0,T]}\\
		+ \lp{( \ud \phi, r^{-1}\sd \phi)|_{t=0}}{L^2(drdV_{\mathbb{S}^2})} + 
		\epsilon^{-1}\lp{\chi_{[r_1,\infty)}w_{ln}\Box^0\phi}{LE^*_0[0,T]}
		+ \epsilon \lp{\chi_{[r_1,\infty)} (w_{ln})^{-1} \partial_t \phi }{LE_0[0,T]} \ . \notag
\end{multline}
By itself this estimate is not yet of the form \eqref{dt_LE}. However, upon combination with
$(r_M-r_1)^\frac{1}{2}\times$\eqref{general_red_shift} the   future boundary 
terms at $t=T$ and $r=r_1$ on the RHS of the last line above can
be traded for  the space-time error $(r_M-r_1)^\frac{1}{2} \lp{\chi_{[r_3,r_\mathcal{T}]}(r-r_\mathcal{T}) \sd\phi }
{L^2(drdV_{\mathbb{S}^2}dt)[0,T]}$ which then yields the desired result.
\end{proof}


\subsection{Proof of the main component estimate \eqref{exterior_LE}}

 This estimate is most easily 
recast  in terms of the Regge-Wheeler type coordinates of Lemma \ref{exterior_lemma}.
We define the space $LE_{RW}$ and $E_{RW}$ by:
\begin{align}
		\lp{\phi}{LE_{RW}[0,T]} \ &= \ \sup_{ j\geqslant 0}\lp{2^{-\frac{1}{2}j}\chi_{ j} \phi}
		{L^2(dr_* dV_{\mathbb{S}^2}ds)[0,T]} \ , \notag\\
		E_{RW}(\phi[s_0]) \ &= \ \lp{\big(\partial_s \phi, \partial_{r_*}\phi, r^{-\frac{3}{2}}(r-r_M)^\frac{1}{2}
		\sd\phi\big)|_{s=s_0}}{L^2(dr_* dV_{\mathbb{S}^2})}^2 \ , \notag
\end{align}
where $\chi_j$ for $j\geqslant 1$ cuts off smoothly where $|r_*|\approx  2^j$, and 
where $\chi_0$ is supported where $|r_*|\lesssim 1$.
For this section all estimates will be in terms of a rescaled 
version of $\Box^0$,   which in Regge-Wheeler coordinates is:
\begin{equation}
		\Box^0_{RW} \ =\  -\partial_s^2 + \partial_{r^*}^2 + V(r_*)\sDelta \ = \ 
		-\la\partial_t ,\partial_t\ra_h \Box^0 \ , \qquad
		V \ = \ -r^{-2} \la\partial_t ,\partial_t\ra_h  \ . \notag
\end{equation}
Then our main estimate  here reads:

\begin{proposition}[Local energy decay estimates for $\Box^0_{RW}$]\label{RW_local_smoothing _prop}
Let $\phi$ be a smooth function which vanishes for $r_*\to \pm \infty$. Then one has
the uniform bound:
\begin{multline}
		\lp{ \big( \partial_{r_*}\phi,(w_{ln})^{-1}\partial_s\phi\big)  }{LE_{RW}[0,T]} + 
		\lp{ (w_{ln})^{-1}r^{-2}(r-r_M)^\frac{1}{2}\sd\phi }{L^2[0,T]}\\
		\lesssim \ E^\frac{1}{2}_{RW}(\phi[0]) + 
		\lp{w_{ln} \Box^0_{RW}\phi}{LE^*_{RW}[0,T]} \ , \label{RW_LE}
\end{multline}
where $w_{ln}=(1+ \big| \ln |r_*| \big|)/\big| \ln (2+ |r_*|) \big|$. 

In addition, if $\overline{\phi}$
is any spherically symmetric function, 
one has the following estimate uniform in $r_*^0<-1$:
\begin{multline}
		\lp{ \chi_{[r_*^0,\infty)}(\partial_s\overline{\phi}, 
		\partial_{r_*}\overline{\phi}) }{LE_{RW}[0,T]} \ \lesssim \ 
		\lp{ \chi_{[2r_*^0,r_*^0]}(\partial_s\overline{\phi}, 
		\partial_{r_*}\overline{\phi}) }{LE_{RW}[0,T]}\\
		+ E^\frac{1}{2}_{RW,r_*>2 r_*^0}(\overline{\phi}[0]) + 
		\lp{\chi_{[2 r_*^0,\infty)} \Box^0_{RW}\overline{\phi}}{LE^*_{RW}[0,T]} \ , \label{sph_RW_LE}
\end{multline}
\end{proposition}\ret

Before giving a proof we use these estimate to establish \eqref{exterior_LE}.

\begin{proof}[Proof that estimates \eqref{RW_LE} and \eqref{sph_RW_LE}
implies \eqref{exterior_LE}]
Given a solution of $\Box^0\phi=G$, will treat the spherically symmetric part $\overline{\phi}$
separately from the non-spherical part $\phi-\overline{\phi}$.

\step{1}{Control of the spherical part}
That \eqref{sph_RW_LE} implies \eqref{exterior_LE} for $\overline{\phi}$
is an immediate consequence of 
the coordinate asymptotics in parts i) and ii) of Lemma \ref{exterior_lemma}. First notice 
that by the asymptotic formulas of part i) of Lemma \ref{exterior_lemma}
the null vector-field $\partial_s + \partial_{r_*}$ is regular and non-vanishing
up to $r=r_M$ in the original $(t,r)$ coordinates. Therefore we have that:
\begin{equation}
		L \ = \ \partial_s + \partial_{r_*} \ , \qquad\qquad
		\bL \ = \ \Omega^{-2}(\partial_s-\partial_{r_*}) \ 
		= \ O\big(r(r-r_M)^{-1}\big)(\partial_s-\partial_{r_*}) \ , \notag
\end{equation}
is a regular null pair throughout  $r>r_M$.
Next, note that one has:
\begin{equation}
		dsdr_* \ = \ \frac{dr_*}{dr}dtdr \ = \ O\big(r(r-r_M)^{-1}\big)dtdr  \ = \ 
		O\big(\Omega^{-2} \big)dtdr \ . \notag
\end{equation}
Combining this with the asymptotics
$r_*=A\ln (r-r_M)+O(1)$ as $r_*\to-\infty$ and $r_*=r+B\ln(r)+O(r^{-1})$ as $r_*\to \infty$,
we see that:
\begin{align}
		\lp{w_k\chi_{[r_M,r_\mathcal{T}]} \psi}{LE_0 } \ &\lesssim \ 
		\lp{\chi_{(-\infty,0]} \psi}{LE_{RW} } \ \lesssim \  \lp{w_{-k}\chi_{[r_M,r_\mathcal{T}]} \psi}{LE_0 }
		\ , \notag\\
		\lp{\chi_{[r_\mathcal{T},\infty)} \psi}{(LE_0,LE^*_0) } \ 
		&\approx \  \lp{\chi_{[0,\infty)} \psi}{(LE_{RW},LE^*_{RW}) } 
		\ , \qquad
		\lp{\chi_{(-\infty,0]} \psi}{LE_{RW}^* } \ 
		\lesssim \  \lp{w_{-k}\chi_{[r_M,r_\mathcal{T}]} \psi}{LE^*_0 } \ , 
		\notag
\end{align}
where the implicit constants depend on $k>1$. In addition note that the logarithmic part of $w_k$
is essentially constant when $r_*\approx -2^j$, so it may be removed from the integral 
in the first term on the RHS of \eqref{sph_RW_LE} above.
Finally, note that by standard local estimates one can control $E_{RW,r_*>2 r_*^0}(\phi[0])$
in terms of the standard energy $E_{r>r_M}(\phi[0])$.

\step{2}{Control of the non-spherical part}
For this we use exactly the same estimates as above with a slight twist. We claim that for $C>0$
sufficiently large one has the analogous bound to \eqref{sph_RW_LE}:
\begin{multline}
		\lp{\chi_{[r_*^0,\infty)} \big( \partial_{r_*}(\phi-\overline{\phi}),
		(w_{ln})^{-1}\partial_s( \phi-\overline{\phi}) \big)}{LE_{RW}[0,T]} + 
		\lp{\chi_{[r_*^0,\infty)} (w_{ln})^{-1} r^{-2}(r-r_M)^\frac{1}{2}
		\sd(\phi-\overline{\phi}) }{L^2[0,T]}\\
		 \ \lesssim \ 
		\lp{ \chi_{[2Cr_*^0,r_*^0]} \partial_{r_*}(\phi-\overline{\phi}) }{LE_{RW}[0,T]}
		+ E^\frac{1}{2}_{RW,r_*>2 C r_*^0}\big( (\phi -\overline{\phi})[0]\big) + 
		\lp{\chi_{[2 Cr_*^0,\infty)} \Box^0_{RW}(\phi-\overline{\phi})}{LE^*_{RW}[0,T]} 
		\ . \label{truncated_LE}
\end{multline}

To prove it first apply \eqref{RW_LE} to the quantity $\chi_{r_*>2Cr_*^0}(\phi-\overline{\phi})$, where 
$\chi_{r_*>2Cr_*^0}\equiv 1$ for $r_*>Cr_0^*$ with standard derivative bounds.
This generates the above estimate  modulo  two errors:
\begin{equation}
		Error_1 \ = \ \lp{\chi_{[2Cr_*^0,Cr_*^0]}\langle r_*\rangle^{-1} (\phi
		-\overline{\phi})(0)}{L^2} \ , \qquad
		Error_2 \ = \ \lp{\chi_{[2Cr_*^0,Cr_*^0]}\langle r_*\rangle^{-1} (\phi-\overline{\phi})}{LE_{RW}[0,T]} 
		\ . \label{truncation_error_line}
\end{equation}
To handle these terms we use the sharp Hardy type estimates \eqref{index_hardy} and 
\eqref{improved_hardy} of the Appendix respectively.  

By an immediate application of \eqref{index_hardy} with $s=1$ and $p=2$, followed by the spherical
Poincare estimate \eqref{S2_poincare} to control the undifferentiated term on the RHS, we have:
\begin{equation}
		Error_1 \ \lesssim \ E^\frac{1}{2}_{RW,r_*>2 C r_*^0}\big( (\phi -\overline{\phi})[0]\big) \ . \notag
\end{equation}
 
To handle $Error_2$ apply \eqref{improved_hardy} with $C =2^{j-i}$ and $s=1$ which yields:
\begin{equation}
		Error_2 \ \lesssim \ C^\frac{1}{2}
		\lp{\chi_{[-2Cr_*^0,-r_*^0]}
		\partial_{r_*}(\phi-\overline{\phi})}{LE[0,T]}+
		C^{-\frac{1}{2}}\big(\lp{\chi_0 \sd\phi }{LE[0,T]} + \lp{\chi_{[-r_*^0,-1]}
		\partial_{r_*}(\phi-\overline{\phi})}{LE[0,T]}\big) \ . \notag
\end{equation}
For $C\gg 1$   the second term of the RHS above may safely be absorbed 
into the LHS of estimate \eqref{truncated_LE}.
\end{proof}

We now give the proofs for Proposition \ref{RW_local_smoothing _prop}.

\begin{proof}[Proof of estimate \eqref{RW_LE}]
In light of  the second estimate \eqref{sph_RW_LE}, which is independently proved below,
we may assume that  $\overline{\phi}=0$. The 
method is to combine the analysis of  \cite{BS_uniform} with \cite{MMTT_sch}.

\step{1}{Proof of a degenerate local energy decay estimate}
First we show that for all $k>1$:
\begin{multline}
		\lp{\langle r_*\rangle^{-\frac{1}{2}k} \partial_{r_*} \phi}{L^2[0,T]} +
		\lp{ (w_1)^{-1} r^{-2} (r-r_M)^\frac{1}{2}\sd\phi }{L^2[0,T]}\\ 
		\lesssim \ C^{-1}\lp{\partial_{r_*} \phi }{LE_{RW}[0,T]}+
		\sup_{0\leqslant s\leqslant T} 
		E^\frac{1}{2}_{RW}(\phi[s]) + C \lp{ \Box^0_{RW}\phi }{LE^*_{RW}[0,T]} \ , \label{degen_RW}
\end{multline}
where the implicit constant depends on $k$ but is uniform in $C>0$, and 
$w_1=r(r-r_\mathcal{T})^{-1}$.

For a $C^1(\mathbb{R})$  (but not necessarily smooth) function $a(r_*)$ we form the operator 
$A=op^w(ia\xi_{r_*})=a\partial_{r_*}+\frac{1}{2}a'$. Then for test functions $\phi$ we 
integrating the quantity $\Box^0_{RW}\phi \cdot A\phi$ over the space-time slab $0\leqslant  t \leqslant T$.
After some integration by parts we arrive at the Virial type identity
(see Section 3  of \cite{BS_uniform} for further details):
\begin{equation}
		\int_0^T\!\!\!\!\int_{\mathbb{R}\times\mathbb{S}^2} \big(
		a' \phi_{r_*}^2 - \frac{1}{4}a'''\phi^2 - \frac{1}{2}a V' |\sd\phi|^2
		\big)dr_* dV_{\mathbb{S}^2}ds  =  
		-\int_{\mathbb{R}\times\mathbb{S}^2}\!\!\!\!\!\! \partial_s\phi A\phi
		dr_* dV_{\mathbb{S}^2}\Big|_0^T -
		\int_0^T\!\!\!\!\int_{\mathbb{R}\times\mathbb{S}^2}\!\!\!\!\!
		\Box^0_{RW}\phi A\phi dr_* dV_{\mathbb{S}^2}ds \ . \notag
\end{equation}
Here $a'''$ is interpreted in the sense of distributions if necessary. If we choose $a$ so that:
\begin{equation}
		\tag{Assumption 1} 
		  a(r_*) \ \approx \ \frac{r_*}{1+|r_*|} \ , 
		  \qquad\qquad a'(r_*) \ \lesssim \ \frac{1}{1+|r_*|} \ ,
		  \notag
\end{equation}
then  using the Hardy/Poincare  estimate \eqref{index_hardy} with $s=1$ and $p=2$
 the boundary terms at $t=0,T$ are estimated 
by the energy. In addition, the non-negative
weight $-aV'$ is then consistent with estimate \eqref{degen_RW}.
Thus, by using Young's inequality to split the last term in the previous identity above, 
and then estimate \eqref{index_hardy} with $s=\frac{3}{2}$ and $p=\infty$ to bound $a'\phi$ 
in terms of $\partial_{r_*}\phi$
and a local $\sd\phi$ contribution (which may be absorbed to the LHS), we have:
\begin{multline}
		\int_0^T\!\!\!\!\int_{\mathbb{R}\times\mathbb{S}^2} \big(
		a' \phi_{r_*}^2 - \frac{1}{4}a'''\phi^2 \big)dr_* dV_{\mathbb{S}^2}ds
		+ \lp{r^{-3}(r-r_\mathcal{T})
		(r-r_M)^\frac{1}{2}\sd\phi }{L^2[0,T]}^2 \\ 
		\lesssim \ 
		C^{-2}\lp{\partial_{r_*} \phi }{LE_{RW}[0,T]}^2+
		\sup_{0\leqslant s\leqslant T}E(\phi[s]) + C^2
		\lp{ \Box^0_{RW}\phi }{LE^*_{RW}[0,T]}^2
		\ . \label{rough_RW_start}
\end{multline}

To deal with the $\partial_{r_*}\phi$ terms on the LHS of \eqref{rough_RW_start}
we need to further constrain $a$. To see what is the correct choice, first note that one
has the following Hardy type inequality for smooth $a$ and fixed $\lambda>0$:
\begin{equation}
		\int \Big(a'''-\lambda \frac{(a'')^2}{a'}\Big)\phi^2dr_* \ \leqslant \ 
		\lambda^{-1} \int a' \phi_{r_*}^2 dr_* \ . \label{funny_hardy}
\end{equation}
This  follows at once from integrating  $\partial_{r_*}(a''\phi^2)$ and then applying Young's
inequality to the term containing $\phi_{r_*}$. From this we see that an ideal set of assumptions 
would be:
\begin{equation}
		\tag{Assumption 2} 
		a' \ \approx \ \langle r_*\rangle^{-k}  \ , \notag
\end{equation}
\begin{equation}
		\tag{Assumption 3}
		\frac{(a'')^2}{a'} \ \leqslant \ \gamma  a'''  \ , \qquad  0 < \gamma < 1 \ . \notag
\end{equation}
Taken together, and setting $\lambda^{-1}=2\gamma$ in \eqref{funny_hardy}, these assumptions
 would give the lower bound:
\begin{equation}
		c_{\gamma} \lp{\langle r_*\rangle^{-\frac{1}{2}k}\partial_{r_*}\phi}{L^2(dr_*)}^2 \ \leqslant \ 
		(1-\gamma)\int a' \phi_{r_*}^2 dr_* \ \leqslant \ \int(a'\phi_{r_*}^2 - \frac{1}{4}a'''\phi^2)dr_* \ ,
		\label{main_lower}
\end{equation}
which would be sufficient for our purposes. 

Unfortunately it turns out that Assumptions 2 and 3 as stated are
inconsistent with each other due to inflections of $a'$. A natural way around this
is to eliminate inflection points by working with a singular weight $a(r_*)$.
In other words we trade Assumption 3  for the modification:
\begin{equation}
		\tag{Assumption 3'}
		a'(r_*) \ = \ \frac{1}{(1+|\epsilon r_*|)^k} \ , \notag
\end{equation}
in which case the original Assumption 3 holds for $r_*\neq 0$ with $\gamma =\frac{k}{k+1}$, and 
\eqref{funny_hardy} still holds in the sense of distributions. The difference now is that
$a'''$ contains a delta function, so
\eqref{main_lower} must be replaced by:
\begin{equation}
		(1-\gamma)\int a' \phi_{r_*}^2 dr_* 
		\ \leqslant \ \int(a'\phi_{r_*}^2 - \frac{1}{4}a'''\phi^2)dr_* + \epsilon k\phi^2|_{r_*=0} \ . \notag
\end{equation}
The last term on the RHS is easily eliminated by the assumption that $\overline{\phi}=0$ which gives:
\begin{equation}
		\int_{\mathbb{S}^2} \phi^2|_{r_*=0} dV_{\mathbb{S}^2} \ \lesssim \ 
		\lp{\chi_0( \partial_{r_*}\phi,r_*\sd\phi) }{L^2(dr_*dV_{\mathbb{S}^2})}^2 \ . \notag
\end{equation}
For $\epsilon\ll 1$ sufficiently small this allows us to absorb $\phi^2|_{r_*=0}$ and we arrive at
 \eqref{degen_RW}.

\step{2}{Improvement of the weights on $\partial_{r_*}\phi$ as $r_*\to \pm \infty$}
Next, we improve estimate \eqref{degen_RW} to the following:
\begin{equation}
		\lp{ \partial_{r_*} \phi}{LE_{RW}[0,T]} +
		\lp{ (w_1)^{-1} r^{-2} (r-r_M)^\frac{1}{2}\sd\phi }{L^2[0,T]}\
		\lesssim \  		\sup_{0\leqslant s\leqslant T} 
		E^\frac{1}{2}_{RW}(\phi[s]) +  \lp{ \Box^0_{RW}\phi }{LE^*_{RW}[0,T]} \ . \label{degen_RW'}
\end{equation}

To prove this we  integrate the quantities $\Box_{RW}^0\phi \cdot a_j^\pm \partial_{r_*}\phi$
over the slabs $0\leqslant s\leqslant T$, where we choose $(a_j^\pm)'=\pm 2^{-j}\chi_j\chi_{\pm r_*> 1}$ 
for $j\geqslant 1$
and $a_j^\pm\equiv 0 $ in a neighborhood of $r_*=0$.
 This  gives the bound:
\begin{multline}
		\lp{ \chi_{|r_*|>2}\partial_{r_*} \phi}{LE_{RW}[0,T]} \ \lesssim \ 
		\lp{ \chi_{|r_*|>1} r^{-2} (r-r_M)^\frac{1}{2}\sd\phi }{L^2[0,T]}\\
		+ \lp{\chi_{|r_*|<1}\partial_{r_*}\phi}{L^2[0,T]} + 
		\sup_{0\leqslant s\leqslant T}E^\frac{1}{2}(\phi[s]) +
		\lp{ \Box^0_{RW}\phi }{LE^*_{RW}[0,T]} \ , \notag
\end{multline}
which can be added safely into \eqref{degen_RW} in order to produce \eqref{degen_RW'}.

\step{3}{Removal of the future energy and addition of $\partial_s\phi$}
Our goal here is the more complete estimate:
\begin{multline}
		\lp{ \partial_{r_*} \phi}{LE_{RW}[0,T]} +
		\lp{(w_1)^{-1} \partial_s \phi}{LE_{RW}[0,T]}+
		\lp{ (w_1)^{-1} r^{-2} (r-r_M)^\frac{1}{2}\sd\phi }{L^2[0,T]}\\
		\lesssim \  	
		E^\frac{1}{2}_{RW}(\phi[0]) +  \lp{w_1 \Box^0_{RW}\phi }{LE^*_{RW}[0,T]} \ . \label{degen_RW''}
\end{multline}

To prove this bound first note that by the usual calculations 
used to produce the conservation of energy one has:
\begin{equation}
		\sup_{0\leqslant s\leqslant T}E_{RW}(\phi[s]) \ \lesssim \ 
		 \lp{w_1 \Box^0_{RW}\phi }{LE^*_{RW}[0,T]} \lp{(w_1)^{-1} \partial_s \phi}{LE_{RW}[0,T]} \ . \notag
\end{equation}
On the other hand, to control the quantity $\partial_s\phi$ one  integrates
$2^{-j}(w_1)^{-2}\chi_j \phi \Box^0_{RW}\phi$ over the slabs $0\leqslant s\leqslant T$. After
some integration by parts, using Poincare/Hardy type estimates to bound undifferentiated
terms and boundary contributions, and then supping over $j\geqslant 0$,
one has:
\begin{multline}
		\lp{(w_1)^{-1} \partial_s \phi}{LE_{RW}[0,T]} \ \lesssim \ 
		\lp{ \partial_{r_*} \phi}{LE_{RW}[0,T]} +
		\lp{ (w_1)^{-1} r^{-2} (r-r_M)^\frac{1}{2}\sd\phi }{L^2[0,T]}\\ 
		+  \sup_{0\leqslant s\leqslant T} 
		E^\frac{1}{2}_{RW}(\phi[s]) +\lp{ \Box^0_{RW}\phi }{LE^*_{RW}[0,T]} \ . \notag
\end{multline}
Adding the last two lines together with \eqref{degen_RW'} yields \eqref{degen_RW''}.

\step{4}{Improvement of the local weight}
Our final task is  to establish \eqref{degen_RW''} with the local weight $w_1$
replaced by the weight $w_{ln}$ as written in estimate \eqref{RW_LE} above. To do this
we  recall the main local estimate of \cite{MMTT_sch} which can be stated in the following form:

\begin{proposition}\label{Daniels_prop}
The following two statements hold:
\begin{enumerate}[a)]
	\item  Let $\phi$ be a space-time function  which is supported in the region
	$|r_*|\leqslant C$ for some fixed $C>0$. Then one has the bound:
	\begin{equation}
		\lp{\partial_{r_*} \phi}{L^2[0,T]} +
		\lp{(w_{ln})^{-1}(\partial_s \phi,\sd\phi)}{L^2[0,T]} \ \lesssim \ 
		E^\frac{1}{2}(\phi[0]) + \lp{w_{ln}  \Box^0_{RW}\phi }{L^2[0,T]} \ . \notag
	\end{equation} 
	\item  Let $f$ be  a space-time supported in the region
	$|r_*|\leqslant C$ for some fixed $C>0$ with the property that
	$w_{ln}f\in L^2[0,T]$. Then there exists a function $\phi$ defined for
	$s\in[0,T]$ and supported
	in the region $|r_*|\leqslant 2C$ such that $\Box^0_{RW}\phi-f$ is supported away from $|r_*|\leqslant 
	\frac{1}{2}C$ and such that:
	\begin{equation}
			\sup_{0\leqslant s\leqslant T} E^\frac{1}{2}(\phi[s])
			+ \lp{\partial_{r_*} \phi}{L^2[0,T]}
			+\lp{(w_{ln})^{-1}(\partial_s\phi,\sd\phi) }{L^2[0,T]}
			+ \lp{ \Box^0_{RW}\phi-f }{L^2[0,T]} \! \lesssim \! 
			\lp{w_{ln}f}{L^2[0,T]} \ . \label{dans_truncation_est}
	\end{equation}
\end{enumerate}
\end{proposition}
For a proof of these bounds refer to parts a) and b) of Proposition 3.3 in \cite{MMTT_sch}. Note that 
in both cases the analysis is purely local and relies only on the structural assumption
that the potential $V(r_*)$ in the definition of $\Box^0_{RW}$ has a unique positive
non-degenerate local maximum at $r_*=0$. 

To show that parts a) and b) of Proposition \ref{Daniels_prop} imply \eqref{degen_RW} with local weight
$w_{ln}$ involves a series of steps. First, one employs 
part b) to produce a $\td{\phi}$ such that
$\Box_{RW}^0\td{\phi}-\chi_0 \Box_{RW}^0\phi$ is supported away from $r_*=0$ with good $L^2$ bounds, 
and such that \eqref{degen_RW''} holds for $\td{\phi}$ with $w_1$ replaced by $w_{ln}$. Then one may apply
\eqref{degen_RW''} to the remainder $\phi-\td{\phi}$ to show that 
\eqref{degen_RW''} holds for $\phi$ with RHS norm containing $w_{ln}$ instead of 
$w_1$. Finally, one replaces the weight $(w_1)^{-1}$ on the LHS of this last bound 
with $(w_{ln})^{-1}$ by applying
part a) above to $\chi_0\phi$.
 
This concludes the proof of \eqref{RW_LE}.
\end{proof}

\begin{proof}[Proof of estimate \eqref{sph_RW_LE}]
Note that the spherical part of $\Box_{RW}^0$ is simply the 1D wave equation
$\Box=-\partial_s^2+\partial_{r_*}^2$, for which it is easy to produce estimates
of the form  \eqref{sph_RW_LE} via the multipliers $\chi_{(r_*^0,\pm 2^j)}\partial_{r_*}$, 
where $r_*^0\leqslant -2^j$. This is analogous to \textbf{step 2} and \textbf{step 3} of the
previous proof. Further details are left to the reader. 
\end{proof}

\section{Local Energy Decay for Spin-Zero Fields Part II: Proof of the Main 
Estimate}\label{LE_section2}

Our main goal here is to prove Theorem \ref{main_LE_thm1} which we state
again here for the convenience of the reader.

\begin{theorem}[Inverse angular gradient local energy decay estimates]\label{Main_LE}
Let $\phi(t,r,x^A)$ and $H(t,r,x^A)$ be scalars,  and let $G_{a}(t,r,x^A)$ be a one form 
in $(t,r)$ variables which also depends on the $\mathbb{S}^2$ coordinates. Suppose
all of these quantities have no radial component::
\begin{equation}
		\int_{\mathbb{S}^2}\phi(t,r) dV_{\mathbb{S}^2} \ = \ 
		\int_{\mathbb{S}^2}H(t,r) dV_{\mathbb{S}^2}  \ = \ 
		\int_{\mathbb{S}^2}G(t,r) dV_{\mathbb{S}^2}  \ \equiv \ 0 \ , \notag
\end{equation}
throughout the region $[0,T]\times[r_0,\infty)$.
Finally, set ${\star}_h d G=K$. If $\phi, G, H$ are all supported in $\{r \leqslant CT\}$  and
\begin{equation}
		\Box^0\phi \ = \ \nabla^a G_a + H \ , \label{spin_zero_eq'}
\end{equation}
then one has the estimate:
\begin{equation}
\begin{split}
		\lp{(w_{ln})^{-1} \big(\ud(-\sDelta)^{-\frac{1}{2}}\phi,r^{-1}\phi\big)}{LE_0[0,T]}
		 \lesssim  & \
		\lp{(-\sDelta)^{-\frac{1}{2}}\big(\ud \phi(0)- G(0)\big)}{L^2(drdV_{\mathbb{S}^2})}
		+\lp{r^{-1}\phi(0)}{L^2(drdV_{\mathbb{S}^2})} \\
	&	+ \lp{w_{ln} \big(
		r^{-1}G,(-\sDelta)^{-\frac{1}{2}}K, (-\sDelta)^{-\frac{1}{2}}H\big)}{LE^*_0[0,T]}
		\ . 
\end{split}\label{main_est}
\end{equation}
\end{theorem}\ret

The main step in the proof is the following elliptic result:
\begin{proposition}\label{p:ell}
Let $G$ be as in the theorem. Then there exists a function $ \aphi$ 
in  $[0,T]\times[r_0,\infty)$, with no radial component,  satisfying the following estimates:
\begin{equation}\label{dsolve}
\begin{split}
	\lp{ r^{-1} \big(\ud\aphi,r^{-1}\aphi,r^{-1} \sd \aphi\big)}{LS^*_0[0,T]}  	 
+ \lp{ (-\sDelta)^{-\frac{1}{2}} \nabla_{r,t}(G-\ud \aphi) }{LS^*_0[0,T]}
 \lesssim  \lp{\big(r^{-1}G,(-\sDelta)^{-\frac{1}{2}}K \big)}{LS^*_0[0,T]}  ,
\end{split}
\end{equation}
\begin{equation}\label{dsolve1}
\begin{split}
	\lp{ w_{ln} r^{-2} \sd \aphi}{LS^*_0[0,T]}  	 
+ \lp{ w_{ln} (-\sDelta)^{-\frac{1}{2}} \nabla_{r,t}(G-\ud \aphi) }{LS^*_0[0,T]}
 \lesssim  \lp{w_{ln} \big(r^{-1}G,(-\sDelta)^{-\frac{1}{2}}K \big)}{LS^*_0[0,T]}  .
\end{split}
\end{equation}
\end{proposition}

We first use the proposition to prove Theorem~\ref{Main_LE}. This is
achieved by applying the bound \eqref{box0_LE} to the function $\psi =
\Delta^{-\frac12}(\phi-\aphi)$.  The left hand side of
\eqref{main_est} is controlled by the left hand side of
\eqref{box0_LE} applied to $\psi$ and the first term on the left hand
side of \eqref{dsolve}.  It remains to estimate the right hand side of
\eqref{box0_LE} applied to $\psi$.

For the energy part $E^\frac12(\psi[0])$ we use the trace estimate \eqref{trace_est}
of Lemma \ref{HPT_Lemma} below. Note  it is precisely here that  truncation
to $r\leqslant CT$ is used. Estimating 
$\aphi$ and $\ud \aphi - G$ in this way we obtain from \eqref{dsolve}
the following bound:
\[
	\lp{(-\sDelta)^{-\frac{1}{2}}\big( \ud\aphi(0)- G(0)\big)}{L^2(drdV_{\mathbb{S}^2})}
+\lp{r^{-1}\aphi(0)}{L^2(drdV_{\mathbb{S}^2})} 
 \lesssim  \lp{\big(r^{-1}G,(-\sDelta)^{-\frac{1}{2}}K \big)}{LE^*_0[0,T]} \ ,
\]
which is combined with  the energy component
of the right hand side of \eqref{main_est}. 
It remains to estimate $\|\Box^0 \psi\|_{LE_0^*}$, for which we
compute:
\[
\begin{split}
\Box^0 \psi \ =& \  (-\sDelta)^{-\frac12} ( \nabla^a G_a + H) -  (-\sDelta)^{-\frac12} (\nabla^a \nabla_a + r^{-2}
\sDelta) 
\aphi \ ,
\\
\ =& \  (-\sDelta)^{-\frac12}  \nabla^a (G_a - \nabla_a \aphi)  + 
r^{-2} (-\sDelta)^{\frac12} \aphi +  (-\sDelta)^{-\frac12} H \ .
\end{split}
\]
and use the bound \eqref{dsolve1} for the first two terms.

 \begin{proof}[Proof of Proposition~\ref{p:ell}]
   We observe that the proposition has no reference to the metric  $h$. We
   begin with several simplifications. 

\sstep{1} Using a radial partition of unit, we reduce the
   problem to the case when we know that $G$ is supported in a dyadic
   annulus $ \{r \approx R\}$, $R \gtrsim 1$,  and we seek $\apsi$ with support in a
   slightly enlarged annulus.

\sstep{2}  Secondly, given $R \gtrsim 1$, we use a partition of
   unit in time on the $R$ scale to reduce the problem to the case
   when $G$ is supported in a dyadic annulus $ \{r \approx R\}$ and a
   time interval of size $\delta T \approx R$ and we seek $\apsi$ with
   support in a similar but slightly enlarged region.


\sstep{3} We rescale the problem to the unit scale in
   $t$ and $r$. Then matters are   reduced to the next
   lemma.
   
\end{proof}

\begin{lemma}
  Let $G$ be a one form in $(t,r)$ variables in $[0,\infty) \times
  [0,\infty) \times \S^2$  which is  supported in $[0,1] \times [0,1] \times
  \S^2$. Then there exists a function $\apsi$ in $[0,\infty) \times
  [0,\infty) \times \S^2$, supported in $[0,2] \times [0,2] \times
  \S^2$, so that the following estimates hold:
\begin{align}
	\| \aphi\|_{H^1} + \|  (-\sDelta)^{-\frac12} \nabla_{r,t} (G - \ud \aphi)\|_{L^2}
	\ &\lesssim \ \|G\|_{L^2} + \|  (-\sDelta)^{-\frac12} \ud G\|_{L^2} \ , \label{aphi-noln}\\
	\| w_{ln}  (-\sDelta)^{\frac12}\aphi\|_{L^2} + \| w_{ln} (-\sDelta)^{-\frac12} \nabla_{r,t} (G - \ud \aphi)\|_{L^2}
	\ &\lesssim \ \|w_{ln}G\|_{L^2} + \| w_{ln} (-\sDelta)^{-\frac12} \ud G\|_{L^2} \ . \label{aphi-ln}
\end{align}
where the logarithmic weight $w_{ln}$ is now centered at $r = \frac12$.
\end{lemma}

We remark that the lemma includes both the case when $G$ is compactly
supported, and the cases when we have either a $t = const$ boundary or
an $r = const$ boundary or both.  

\begin{proof}
In all cases, we reduce the problem
to the boundaryless case by taking suitable even/odd extensions of the
components of $G$ to the other three quadrants in the $(r,t)$ plane.
 Now we have the $(t,r)$ 1-form $G$ with variables $\R \times \R  \times
\S^2$ supported in $[-1,1] \times [-1,1] \times \S^2$. 

Using the euclidean metric in $(t,r)$ we define $\aphi$ by:
\[
\aphi = \chi(r,t) \aphi_0 \ , \qquad \aphi_0\  = \ (- \Delta_{r,t} - \sDelta)^{-1} \ud^\star G  \ ,
\]
where $\chi$ is a smooth cutoff supported in $[-2,2] \times [-2,2]$ 
and which equals $1$ in $[-1,1] \times [-1,1]$.

Then:
\[
G- \ud\aphi \ = \ \chi G_0 -  \aphi d\chi \ , 
\]
where:
\[
\begin{split}
 \td{G} \ =& \  G- \ud\aphi_0  \ = \ 
 (1 + \Delta_{r,t}(- \Delta_{r,t} - \sDelta)^{-1}) G + (- \Delta_{r,t} - \sDelta)^{-1} \ud^\star \ud G \ ,
\\
\ = & \  - \sDelta(- \Delta_{r,t} - \sDelta)^{-1} G + (- \Delta_{r,t} - \sDelta)^{-1} \ud^\star \ud G \ .
\end{split}
\]

Then we have the straightforward bounds:
\begin{align}
	\|   \aphi_0\|_{H^1} \ &\lesssim\  \|G\|_{L^2} \ , \notag\\
	\|(-\sDelta)^{-\frac12} \td{G}\|_{H^1} \ &\lesssim  \
	\|G\|_{L^2}+\|(-\sDelta)^{-\frac12} \ud G\|_{L^2} \ . \notag
\end{align}
and the estimate \eqref{aphi-noln} follows after truncation.

For the bound \eqref{aphi-ln} we need the following weighted versions
of the above estimates:
\begin{align}
	\|  w_{ln} (-\sDelta)^{\frac12} \aphi_0\|_{L^2} \ &\lesssim \ \|w_{ln} G\|_{L^2} \ , \notag\\
	\|w_{ln} (-\sDelta)^{-\frac12} \nabla_{r,t} \td{G}\|_{L^2} 
	\ &\lesssim\  \|w_{ln}G\|_{L^2}+\|w_{ln}(-\sDelta)^{-\frac12} \ud G\|_{L^2} \ . \notag
\end{align}
These in turn correspond to the $L^2_{w_{ln}}$ boundedness of the
singular integral operators $\sDelta(- \Delta_{r,t} - \sDelta)^{-1}$, respectively
$(-\sDelta)^{\frac12}\nabla_{r,t} (- \Delta_{r,t} - \sDelta)^{-1}$.
But this can be easily handled within the
context of $A_p$ weights, see for example Chapter V of \cite{Stein}.
\end{proof}

\ret

\section{Appendix: Some Hardy and Poincare Type Estimates}

This section contains some  auxiliary  estimates needed  for  some 
of the proofs in the main body of the paper. They are collected here
for the convenience of the reader.

To state the estimates denote by  $\chi_j$ the indicator function (sharp cutoff)
of $2^j\leqslant |x|\leqslant 2^{j+1}$.
 
\begin{lemma}[Hardy and Poincare  estimates]\label{HPT_Lemma}
Let $\psi$ be a function on $\mathbb{R}\times\mathbb{S}^2$.  Then one has the pair of uniform bounds
for $1\leqslant p\leqslant \infty$ and $0\leqslant i\leqslant j$, and implicit constant
depending on $s>\frac{1}{2}$:
\begin{align}
		\lp{\langle x\rangle^{-s} \psi}{\ell^p L^2 (|x|\leqslant 2^j)} \ &\lesssim \ 
		\lp{\chi_0 \psi }{L^2} +
		\lp{\langle x\rangle^{1-s}\partial_{x} \psi}{\ell^p L^2(|x|\leqslant 2^j) } \ ,
		 \label{index_hardy}\\
		 \lp{\chi_j \langle x\rangle^{-s} \psi}{ L^2  } \ &\lesssim \ 2^{(\frac{1}{2}-s)(j-i)}\big(
		\lp{\chi_0 \psi }{L^2}
		+ \lp{\langle x\rangle^{1-s}\partial_{x} \psi}{\ell^\infty
		L^2(|x|\leqslant 2^i) }\big) \label{improved_hardy}\\ 
		&\ \ \ \ + 2^{(s-\frac{1}{2})(j-i)}\lp{\langle x\rangle^{1-s}\partial_{x} \psi}{\ell^\infty 
		L^2(2^i\leqslant |x|\leqslant 2^{j+1}) }
		 \  .  \notag
\end{align}
Here $\lp{\psi}{\ell^p L^2}^p= \lp{\psi}{L^2(|x|<1)}^p+ \sum_{j\geqslant 0} \lp{\chi_j\psi}{L^2}^p$
(with  standard modification for $p=\infty$),
where the integration is with respect to $dxdV_{\mathbb{S}^2}$.
Similar estimates hold when $\psi$ is also a function
of $t$ and the integrals are with respect to $dxdV_{\mathbb{S}^2}dt$.

In addition, let $G$ be a space-time function (variables $(t,r,x^A)$) defined on the slab
$r\geqslant r_0$ and $0\leqslant t\leqslant T$. Then one has the uniform bound:
\begin{equation}
		\sup_{t\in [0,T]}\lp{G(t)}{L^2(drdV_{\mathbb{S}^2})[r_0,\infty)} \ \lesssim \ 
		\lp{\big( T^{-\frac{1}{2}}G, r^{-\frac{1}{2}}G,
		r^\frac{1}{2}\partial_t G\big)}{L^2(drdV_{\mathbb{S}^2}dt)[0,T]\times [r_0,\infty)}
		\ . \label{trace_est}
\end{equation}
\end{lemma}

\begin{proof}[Proof of estimates \eqref{index_hardy} and \eqref{improved_hardy}]
Set $\chi$ according to  $\chi'=-2^{-k}\chi_k + \chi_0$ and vanishing at infinity.
Then $\chi\equiv 0$ for $|x|\geqslant 2^{k+1}$ and $|x|\leqslant 1$,
and $|\chi|\leqslant 1$ throughout $\mathbb{R}$.
Integrating the quantity $\partial_x(\chi \psi^2)$ with respect to 
$dxdV_{\mathbb{S}^2}$ or $dxdV_{\mathbb{S}^2}dt$, multiplying the result by $2^{(1-2s)k}$,
and taking the square root we have the uniform bound:
\begin{align}
	\lp{\chi_{k} \langle x\rangle^{-s} \psi}{L^2} \ 
		&\lesssim\ 2^{(\frac{1}{2}-s)k}\lp{\chi_0 \psi }{L^2} + 
		2^{(\frac{1}{2}-s)k}\lp{ \psi \partial_{x}\psi }{ L^1(1\leqslant |x|\leqslant 2^{k+1})}^\frac{1}{2} \ , 
		&s&\in \mathbb{R}  \ , 
		\quad k\geqslant 0 \ . \label{basic_hardy}
\end{align}
For the second term on the RHS above we have:
\begin{equation}
		2^{(\frac{1}{2}-s)k}\lp{ \psi \partial_{x}\psi }{ L^1(1\leqslant |x|\leqslant 2^{k+1})}^\frac{1}{2}
		\ \lesssim \ \sum_{0\leqslant k'\leqslant k} 2^{(\frac{1}{2}-s)(k-k')}
		\lp{\chi_{k'} \la x\ra^{-s} \psi }{ L^2 }^\frac{1}{2}
		\lp{\chi_{k'} \la x\ra^{1-s} \partial_x \psi }{ L^2 }^\frac{1}{2} \ . \label{ik_sum}
\end{equation}

The proof of \eqref{index_hardy} follows at once by summing \eqref{basic_hardy} over
$0\leqslant k \leqslant j$ and using then using \eqref{ik_sum}. To bound the RHS of  
\eqref{ik_sum} summed over $k$ we use Young's convolution inequality 
$2^{(\frac{1}{2}-s)|k-k'|}*\ell^p_{k'}\subseteq \ell^p_k$ for $s>\frac{1}{2}$ and then
$\lp{\sqrt{ab}}{\ell^p}\leqslant \epsilon\lp{a}{\ell^p}+\epsilon^{-1}\lp{b}{\ell^p}$ for sequences
$a=(a_k)$ and $b=(b_k)$. This produces an estimate for 
$\lp{\langle x\rangle^{-s} \psi}{\ell^p L^2 (1\leqslant |x|\leqslant 2^j)}$ with upper bound 
RHS\eqref{index_hardy}. To fill in $|x|\leqslant 1$ one uses a standard Sobolev embedding.

It remains to show \eqref{improved_hardy}. Starting with  \eqref{basic_hardy} for $k=j$ followed
by \eqref{ik_sum}, we break the sum of RHS\eqref{ik_sum}  into $0\leqslant k'\leqslant i$ and
$i\leqslant k'\leqslant j$. Following the same procedure of the previous paragraph and using
\eqref{index_hardy} to estimate the resulting undifferentiated term, the sum in the first 
range is bounded by the first term on RHS\eqref{improved_hardy}. For the    second range 
we use the dyadic weight to sum in $\ell^1$ followed by:
\begin{multline}
		\lp{ \la x\ra^{-s} \psi }{\ell^\infty L^2 (2^i\leqslant |x|\leqslant 2^{j+1})}^\frac{1}{2}
		\lp{ \la x\ra^{1-s} \partial_x \psi }{\ell^\infty L^2 
		(2^i\leqslant |x|\leqslant 2^{j+1})}^\frac{1}{2} \\ \leqslant  \ 
		2^{(\frac{1}{2}-s)(j-i)}
		\lp{ \la x\ra^{-s} \psi }{\ell^\infty L^2 (2^i\leqslant |x|\leqslant 2^{j+1})} 
		+ 2^{(s-\frac{1}{2})(j-i)}
		\lp{ \la x\ra^{1-s} \partial_x \psi }{\ell^\infty L^2 
		(2^i\leqslant |x|\leqslant 2^{j+1})} \ .
\end{multline}
For the first term on the RHS above we again use \eqref{index_hardy}.

\end{proof}

\begin{proof}[Proof  of estimate \eqref{trace_est}]
By the pigeon-hole principle there exists $t_0\in [0,T]$ such that:
\begin{equation}
		\lp{G(t_0) }{L^2(drdV_{\mathbb{S}^2})[r_0,\infty)} \ \leqslant \ 
		T^{-\frac{1}{2}} \lp{G}{L^2(drdV_{\mathbb{S}^2}dt)[0,T]\times [r_0,\infty)} \ . \notag
\end{equation}
Then using the fundamental theorem of calculus and integrating on $[t_0,t]$ we have:
\begin{equation}
		\lp{G(t) }{L^2(drdV_{\mathbb{S}^2})[r_0,\infty)}^2 \ \leqslant \ 
		T^{-1} \lp{G}{L^2(drdV_{\mathbb{S}^2}dt)[0,T]\times [r_0,\infty)}^2
		+ 2\Big| \int_{t_0}^t \int_{[r_0,\infty)\times\mathbb{S}^2} 
		G\partial_t G\,  drdV_{\mathbb{S}^2}dt\Big|
		\ . \notag
\end{equation}
Estimate \eqref{trace_est} follows immediately from Young's inequality applied to the last line.
\end{proof}

\end{document}